\documentclass[a4paper,11pt]{article}
\usepackage[french,english]{babel}
\usepackage[babel=true,kerning=true]{microtype}
\usepackage[utf8]{inputenc} 
\usepackage{graphicx}
\usepackage{amssymb}
\usepackage{amsmath}
\usepackage{amsthm}
\usepackage{geometry}
\usepackage{tikz}
\usepackage{hyperref}

\hypersetup{
colorlinks=true,
pdfauthor={Mickaël Kourganoff},
pdftitle={},
pdfcreator=PDFLaTeX,
pdfproducer=PDFLaTeX,
linkcolor=[rgb]{0,0,0.7},
citecolor=[rgb]{0.7,0,0},
}
\usetikzlibrary{shapes.geometric}
\usetikzlibrary{through,calc}
\usetikzlibrary{plotmarks}

\tikzset{
    scale plot marks/.is choice,
    scale plot marks/false/.code={
        \def\pgfuseplotmark##1{\pgftransformresetnontranslations\csname pgf@plot@mark@##1\endcsname}
    },
    scale plot marks/true/.style={},
    scale plot marks/.default=true
}

\newtheorem{prop}{Proposition}[section]
\newtheorem{thm}[prop]{Theorem}
\newtheorem{lemme}[prop]{Lemma}

\newtheorem{cor}[prop]{Corollary}

\theoremstyle{remark}

\newtheorem{remarque}[prop]{Remark}

\theoremstyle{definition}

\newtheorem{definition}[prop]{Definition}
\newtheorem{exemple}[prop]{Example}

\newcommand{\abs}[1]{\left\lvert #1 \right\rvert}

\newcommand{\norm}[1]{\left\lVert #1 \right\rVert}
\newcommand{\setof}[2]{\left\{ #1 ~ \middle\arrowvert ~ #2 \right\}}
\newcommand{\prodscal}[2]{\left\langle #1 ~ \middle\arrowvert ~ #2 \right\rangle}

\begin{document}

\selectlanguage{english}

\title{Similarity structures \\ and de Rham decomposition}
\author{Mickaël Kourganoff}

\maketitle

\abstract{A similarity structure on a connected manifold $M$ is a Riemannian metric on its universal cover $\tilde M$ such that the fundamental group of $M$ acts on $\tilde M$ by similarities. If the manifold $M$ is compact, we show that the universal cover admits a de Rham decomposition with at most two factors, one of which is Euclidean. Very recently, after Belgun and Moroianu conjectured that the number of factors was at most one, Matveev and Nikolayevsky found an example with two factors. When the non-flat factor has dimension $2$, we give a complete classification of the examples with two factors. In greater dimensions, we make the first steps towards such a classification by showing that $M$ is a fibration (with singularities) by flat Riemannian manifolds. Up to a finite covering of $M$, we may assume that these manifolds are flat tori. We also prove a version of the de Rham decomposition theorem for the universal covers of manifolds endowed with locally metric connections. During the proof, we define a notion of transverse (not necessarily flat) similarity structure on foliations, and show that foliations endowed with such a structure are either transversally flat or transversally Riemannian. None of these results assumes analyticity.}

\section{Introduction}

\subsection{Similarity structures}

A similarity $\phi: M_1 \to M_2$ of ratio $\lambda \in \mathbb R_{>0}$ between two Riemannian manifolds $(M_1, g_1)$ and $(M_2, g_2)$ is a diffeomorphism such that $\phi^* g_2 = \lambda^2 g_1$. The similarity group $\mathrm{Sim}(M)$ of a manifold $M$ is the group of all similarities from $M$ to itself.

A similarity structure on a connected manifold $M$ is a Riemannian metric $g$ on its universal cover $\tilde M$ such that $\pi_1(M)$ acts on $\tilde M$ as a subgroup of $\mathrm{Sim}(\tilde M)$: thus, the Riemannian metric is only defined locally ``up to a constant'' on the manifold $M$. Notice that the Levi-Civita connection $\tilde \nabla$ of $g$ does project to a connection $\nabla$ on $M$. Here are three fundamental examples (the first one is a simple particular case of the other two):

\begin{exemple} \label{exFlat}
Consider $N = (\mathbb R^n \setminus \{0\}, g)$, where $g$ is the restriction of the standard Euclidean metric of $\mathbb R^n$, and the subgroup $G$ of $\mathrm{Diffeo}(N)$ generated by the similarity $\varphi: x \mapsto \lambda x$, with some $\lambda \in (0,1)$. The metric $g$ induces a metric on the universal cover of $N$ (which is $N$ itself if $n \geq 3$). Thus, $M = N/G$ is naturally endowed with a similarity structure.
\end{exemple}

\begin{exemple} \label{exClosed}
Let $(M,g)$ be a connected Riemannian manifold, and $(\tilde M, \tilde g)$ its universal cover. Any closed $1$-form $\omega$ on $M$ lifts to an exact $1$-form $\tilde \omega$ on $\tilde M$. Consider a primitive $f$ of $\tilde \omega$ and let $\tilde h = e^f \tilde g$. Then the fundamental group $\pi_1(M)$ acts on $(\tilde M, \tilde h)$ by similarities, and thus $\tilde h$ induces a similarity structure on $M$. The group $\pi_1(M)$ acts by isometries if and only if $\omega$ is exact.
\end{exemple}

\begin{exemple} \label{exCone}
Let $(N,g)$ be any compact connected Riemannian manifold. The \emph{Riemannian cone over $N$} is the manifold $C = N \times \mathbb R_{>0}$ endowed with the Riemannian metric $t^2 g + dt^2$. Consider the subgroup $G$ of $\mathrm{Diffeo}(C)$ generated by the similarity $\varphi: (x, t) \mapsto (x, \lambda t)$ (where $\lambda \in (0,1)$). Then $M = C/G$ is a compact manifold endowed with a similarity structure.
\end{exemple}

In 1979, Gallot studied the holonomy group of Riemannian cones~\cite{MR543217}. The holonomy group of a manifold $M$ endowed with a connection $\nabla$ at a point $x \in M$, written by $\mathrm{Hol}_x(\nabla)$, is the subgroup of $GL(T_xM)$ obtained by the parallel transport along all loops based at $x$. The manifold is said to have irreducible holonomy if there is no subspace of $T_xM$ invariant by the holonomy group at $x$: when $M$ is connected, this property does not depend on the choice of $x$. The holonomy group of a Riemannian manifold is the holonomy group of its Levi-Civita connection. Similarity structures have more possible holonomy groups than Riemannian manifolds. For example, the holonomy group in Example~\ref{exFlat} is $\mathbb Z$.

\begin{thm}[Gallot] \label{thmGallot}
If $(C, g)$ is the Riemannian cone over a compact connected Riemannian manifold $N$, then either $C$ is flat, or it has irreducible holonomy.
\end{thm}

In 2014, Belgun and Moroianu~\cite{MR3491885} asked whether this result generalizes to all similarity structures on compact manifolds. In other words, assuming that a Riemannian manifold $\tilde M$ has a compact quotient $M$ such that $\pi_1(M)$ acts by similarities, but not only by isometries, is it true that $\tilde M$ is either irreducible or flat? In 2015, Matveev and Nikolayevsky~\cite{MR3335002} answered negatively to this question by a counterexample. In this paper, we prove the following result:

\begin{thm} \label{thmConformalStructure}
Consider a compact manifold $M$ endowed with a similarity structure, and its universal cover $\tilde M$ equipped with the corresponding Riemannian structure $g$. Assume that $M$ is not globally Riemannian, \emph{i.e.} $\pi_1(M)$ is not a subgroup of $\mathrm{Isom}(\tilde M)$. Then we are in exactly one of the following situations:
\begin{enumerate}
\item $\tilde M$ is flat.
\item $\tilde M$ has irreducible holonomy and $\dim(\tilde M) \geq 2$.
\item $\tilde M = \mathbb R^q \times N$, where $q \geq 1$, $\mathbb R^q$ is the Euclidean space, and $N$ is a non-flat, non-complete Riemannian manifold which has irreducible holonomy.
\end{enumerate}
\end{thm}

In 2015, Matveev and Nikolayevsky~\cite{matveev2015locally} proved Theorem~\ref{thmConformalStructure} under the assumption that the manifold $(\tilde M, g)$ is analytic, and asked whether the theorem holds without this assumption. Here, we answer positively to this question, by a totally new proof.

\begin{proof}[\textbf{Theorem~\ref{thmConformalStructure} implies Gallot's theorem.}] After admitting that Theorem~\ref{thmConformalStructure} holds, let us prove Theorem~\ref{thmGallot} in a new way. Consider the universal cover $\tilde C$ of the cone $C$ over $N$, and its Cauchy completion $\hat{\tilde C}$. Since $\tilde C$ is the cone over $\tilde N$ (the universal cover of $N$), the difference $\hat{\tilde C} \setminus \tilde C$ is a single point. Since $C$ has a compact quotient endowed with a similarity structure, Theorem~\ref{thmConformalStructure} applies to $C$. To obtain Theorem~\ref{thmGallot}, assume that $\tilde C = \mathbb R^q \times M_1$ where $q \geq 1$, and notice the following contradiction: $\hat {\tilde C} = \mathbb R^q \times \hat{M_1}$, and $\hat{M_1} \setminus M_1 \neq \emptyset$, so the set $\hat{\tilde C} \setminus \tilde C = \mathbb R^q \times (\hat{M_1} \setminus M_1)$ has infinite cardinal. Thus, Theorem~\ref{thmGallot} is proved.
\end{proof}

\begin{proof}[\textbf{Theorem~\ref{thmConformalStructure} generalizes a result obtained by Belgun and Moroianu.}]  In the setting of Theorem~\ref{thmConformalStructure}, Belgun and Moroianu proved that, under an additional assumption on the lifetime of geodesics, only the first two cases are possible (\cite{MR3491885}, Theorem 1.4). We are now going to check that in the third case, this additional assumption cannot be satisfied. Hence, we will see that the proof of Theorem~\ref{thmConformalStructure} contains a new proof of their theorem.

Consider a manifold $M$ endowed with a similarity structure, which satisfies the assumptions of Theorem~\ref{thmConformalStructure}, and its universal cover $(\tilde M, g)$. If $X$ is a vector in the unit tangent bundle $S\tilde M$, denote by $\mathcal L(X)$ the \emph{lifetime} of the half-geodesic tangent to $X$, that is, the supremum of the times for which this half-geodesic is defined. Then, let:
\[ \mu: \tilde M \to [0,+\infty], \quad \mu(x) = \sup \left\{\{0\} \cup \setof{\mathcal L(X)}{X \in S_x\tilde M, \quad \mathcal L (X) < +\infty}\right\}, \]
where $S_x\tilde M$ is the set of unit length vectors tangent to $\tilde M$ at $x$.

Belgun and Moroianu proved that, if $\mu$ is locally bounded on $\tilde M$, then either Case~1 or Case~2 of Theorem~\ref{thmConformalStructure} applies.

To obtain Belgun and Moroianu's theorem from Theorem~\ref{thmConformalStructure}, we assume that the third case holds (\emph{i.e.} $\tilde M = \mathbb R^q \times N$) and look for a contradiction. Let $X \in S\mathbb R^q$ and $Y \in SN$ such that the lifetime $\mathcal L(Y)$ of the half-geodesic tangent to $Y$ in the manifold $N$ is finite. Then for $t \in (0,\pi/2)$, the lifetime of the half-geodesic tangent to $(\cos(t) \cdot X, \sin(t) \cdot Y)$ in the manifold $\tilde M$ is $\mathcal L(Y) / \sin(t)$, which tends to $+\infty$ as $t \to 0$. Thus $\mu$ is not locally bounded and the proof is complete.
\end{proof}

\begin{exemple} \label{exThird}
It is not obvious how to construct examples which fall into the third category in Theorem~\ref{thmConformalStructure}. Let us give the recipe to construct an example:

\begin{enumerate}
\item Choose $q \geq 1$ and consider the torus $\mathbb T^{q+1} = \mathbb R^{q+1} / \mathbb Z^{q+1}$.
\item Consider a linear diffeomorphism of the torus $A \in SL_{q+1}(\mathbb Z)$, such that there exists a number $\lambda \in (0,1)$, and a decomposition $\mathbb R^{q+1} = E^s \oplus E^u$ invariant by $A$, and a positive definite symmetric bilinear form $b$ on $E^s$ satisfying the following:
\begin{enumerate}
\item the stable subspace $E^s$ has dimension $q$, and $A|_{E^s}$ is a similarity of ratio $\lambda$, \emph{i.e.} one may write $A|_{E^s} = \lambda \cdot O$, where $O \in O(E^s, b)$ is a linear mapping which preserves the form $b$,
\item the unstable subspace $E^u$ is one-dimensional.
\end{enumerate}
(In particular, the diffeomorphism $A$ is Anosov.)
\item Construct the mapping torus $M$ of the diffeomorphism $A$ in the following way: take the quotient of $\mathbb T^{q+1} \times (0, +\infty)$ by the mapping $\Phi : (x, z) \mapsto (Ax, \lambda z)$.
\item Consider a basis $(e_1, \ldots, e_q)$ of $E^s$ which is orthonormal for $b$, and $e_{q+1} \in E^u$: the basis $(e_1, \ldots, e_q, e_{q+1})$ of $\mathbb R^{q+1}$ provides local coordinates $(x_1, \ldots, x_{q+1})$ in a neighborhood of each point of $\mathbb T^{q+1}$. Define a Riemannian metric $g$ on $\mathbb T^{q+1} \times (0, +\infty)$ by \[ g = dx_1^2 + \ldots + dx_q^2 + \varphi(z) dx_{q+1}^2 + dz^2, \]
where $\varphi: (0,+\infty) \to (0, +\infty)$ is a smooth function such that for all $z \in (0, +\infty)$, $\varphi(\lambda z) = \lambda^{2q+2} \varphi(z)$.
\end{enumerate}
Notice that $\Phi^*g = \lambda^2 g$: thus the metric $g$ induces a similarity structure on $M$. The universal cover $\tilde M$ of $M$ is isometric to $\mathbb R^q \times N$, where $\mathbb R^q$ is the Euclidean space $(E^s, b)$, and $N = E^u \times (0, +\infty)$. Furthermore, the Gaussian curvature of the surface $N$ is given by $-\frac{\psi''(z)}{\psi(z)}$, where $\psi = \sqrt{\varphi}$. If $N$ is flat, then $\psi''(z)$ is zero everywhere, so $\psi$ is an affine mapping, which contradicts the assumption $\varphi(\lambda z) = \lambda^{2q + 2} \varphi(z)$. Thus, the manifold $N$ is not flat, and $M$ corresponds to the third case of Theorem~\ref{thmConformalStructure}.
\end{exemple}

\begin{remarque} \label{rkMMP} It turns out that in Example~\ref{exThird}, the only possible values for $q$ are $1$ and $2$. Indeed, for $q \geq 3$, Madani, Moroianu and Pilca~\cite{madani2017weyl} proved that it is impossible to construct a linear diffeomorphism $A$ satisfying Conditions (2a) and (2b).
\end{remarque}

We will say that two manifolds $M_1$ and $M_2$ endowed with similarity structures are isomorphic if there is a diffeomorphism between $M_1$ and $M_2$ which lifts to a similarity between the universal covers $\tilde M_1$ and $\tilde M_2$. Matveev and Nikolayevsky's example~\cite{MR3335002} is isomorphic to Example~\ref{exThird} with the choice $q = 1$ and $\varphi(z) = z^4$. In this paper, we prove the following:
\begin{thm} \label{thmDim2}
Consider a manifold $M$ which corresponds to the third case of Theorem~\ref{thmConformalStructure}, and assume that $\dim(N) = 2$. Then $M$ is isomorphic to a manifold constructed in Example~\ref{exThird} for some choice of $q$, $A$ and $\varphi$. In particular, $M$ is the mapping torus of an Anosov diffeomorphism of the torus.
\end{thm}

Theorem~\ref{thmDim2} gives a complete classification of the manifolds which correspond to the third case of Theorem~\ref{thmConformalStructure}, under the assumption that $\dim(N) = 2$. On the other hand, the manifolds corresponding to the first case (\emph{i.e.} flat manifolds) were classified by Fried~\cite{MR604714}.

\bigskip

\noindent\textbf{Fibration by flat tori.} In greater dimensions, the problem of classifying manifolds corresponding to the third case is still open, but we prove the following:

\begin{thm} \label{thmClosures}
In the third case of Theorem~\ref{thmConformalStructure}, consider the foliation $\tilde{\mathcal F}$ induced by the submersion $\tilde M \to N$, and $\mathcal F$ the foliation induced on $M$ by $\tilde{\mathcal F}$. Then $\mathcal F$ is a Riemannian foliation on $M$, and the closures of the leaves form a singular Riemannian foliation $\overline{\mathcal F}$ on $M$, such that each leaf of $\overline{\mathcal F}$ is a smooth manifold of dimension $d$ (which may depend on the leaf) with $q < d < q + n$, where $n = \dim(N)$. Moreover, on each leaf of $\overline{\mathcal F}$, there is a flat Riemannian metric which is compatible with the similarity structure of $M$.
\end{thm}

A Riemannian foliation is a foliation which has a Riemannian structure on its transversal which is compatible with the foliation: the reader may refer to~\cite{MR932463} for the general theory. The situation described in Theorem~\ref{thmClosures} induces a fibration with singularities, where the fibers are the leaves of $\overline{\mathcal F}$ (in particular, there is a dense open set of $M$ which is a nonsingular fibration). To show that $\overline{\mathcal F}$ is a nonsingular fibration as in Theorem~\ref{thmDim2}, one would need to show that $\overline{\mathcal F}$ is nonsingular (\emph{i.e.} all the closures of the leaves of $\mathcal F$ have the same dimension), and that the leaf space of $\mathcal F$ is a smooth manifold. These questions are still open.

The closures of the leaves of a Riemannian foliation are always submanifolds (see~\cite{MR932463}), so the main difficulty in Theorem~\ref{thmClosures} is to prove that the closures of the leaves are flat. Moreover, we show that they have a structure of Riemannian manifold. This implies in particular that the closures of each leaf is finitely covered by a torus (by Bieberbach's theorem: see Theorem~\ref{thmBieberbach}). In fact, we show a more precise result:

\begin{thm} \label{thmTori}
In the setting of Theorem~\ref{thmClosures}, there exists a finite covering $M' \to M$ with the following property: considering the foliation $\mathcal F'$ induced on $M'$ by $\mathcal F$, the closures of the leaves of $\mathcal F'$ are flat tori.
\end{thm}

Thus, $M$ has a finite covering which is a fibration by tori with singularities. We do not know whether it is always possible to choose $M' = M$ in Theorem~\ref{thmTori}. However, in the special case where $q = 1$, the answer is positive, as a consequence of the following:
\begin{thm}[Carrière, 1984, \cite{MR755161}] \label{thmCarriere}
On a compact manifold, if $\mathcal F$ is a foliation of dimension $1$ endowed with a transverse Riemannian structure, then the closures of the leaves are tori.
\end{thm}

\subsection{De Rham decomposition}

A de Rham decomposition of a connected Riemannian manifold $(M, g)$ is a family $(M_0, g_0), (M_1, g_1), \ldots, (M_k, g_k)$ of Riemannian manifolds ($k \geq 0$), where $M_0$ is flat, while $M_1, \ldots, M_k$ are non-flat manifolds which have irreducible holonomy, such that:
\[ (M, g) = (M_0, g_0) \times (M_1, g_1) \times \ldots \times (M_k, g_k). \]

The de Rham decomposition theorem~\cite{MR0052177} states the following (see also~\cite{MR0152974}):

\begin{thm}[de Rham, 1952] \label{thmDeRham}
\begin{enumerate}
\item If a connected Riemannian manifold $(M, g)$ admits a de Rham decomposition, then it is unique up to the order of the factors.
\item (Local version.) Any point of a Riemannian manifold has a neighborhood which admits a de Rham decomposition.
\item (Global version.) Every \emph{complete}, \emph{simply connected}, connected Riemannian manifold admits a de Rham decomposition.
\end{enumerate}
\end{thm}

Notice that the universal cover of a manifold $M$ endowed with a similarity structure (which is not globally Riemannian) is never complete: otherwise, the similarities which are not isometries would have a fixed point (by the Banach fixed point theorem), but $\pi_1(M)$ needs to act freely on $\tilde M$. However, Theorem~\ref{thmConformalStructure} states that $\tilde M$ does admit a de Rham decomposition. More precisely, Theorem~\ref{thmConformalStructure} may be rephrased as follows:

\begin{thm}
If $(\tilde M, g)$ is the universal cover of a compact, connected manifold $M$ endowed with a similarity structure, then $(\tilde M, g)$ admits a de Rham decomposition. Furthermore, the number of factors in the decomposition is at most $2$: if it is exactly two, then one of the factors is the Euclidean space.
\end{thm}

In this paper, we also prove a new version of de Rham's decomposition theorem in a more general framework.
\begin{definition}
A locally metric connection on a manifold $M$ is a torsion-free connection $\nabla$ which lifts to a connection $\tilde \nabla$ on the universal cover such that $\tilde \nabla$ is the Levi-Civita connection of a Riemannian metric $g$. Equivalently, a locally metric connection is a torsion-free connection whose restricted holonomy group $\mathrm{Hol}^0 (\nabla)$ (\emph{i.e.} the subgroup of $\mathrm{Hol} (\nabla)$ given by parallel transport along loops which are homotopic to a constant) is a relatively compact subgroup of $GL_n(\mathbb R)$.
\end{definition}

\begin{exemple}
If $M$ is a manifold endowed with a similarity structure, the Levi-Civita connection $\tilde \nabla$ of the Riemannian metric $g$ on $\tilde M$ induces a locally metric connection $\nabla$ on $M$. The locally metric connections obtained in this way are exactly those which preserve a conformal structure (see~\cite{MR3491885} for more details).
\end{exemple}

Unlike similarity structures, locally metric connections behave well with respect to the product structure: if $(M_1, \nabla_1)$ and $(M_2, \nabla_2)$ are two manifolds endowed with locally metric connections, then the product connection $(\nabla_1, \nabla_2)$ is again a locally metric connection on $M_1 \times M_2$, but this connection is not given by a similarity structure. In this paper, we will prove the following generalization of Theorem~\ref{thmDeRham}:

\begin{thm} \label{thmLocallyMetric}
Consider a compact connected manifold $(M, \nabla)$, where $\nabla$ is a locally metric connection, and a Riemannian metric $g$ on its universal cover $(\tilde M, \tilde \nabla)$ such that $\tilde \nabla$ is the Levi-Civita connection of $g$. Then $(\tilde M, g)$ admits a de Rham decomposition.
\end{thm}

Again, it is important to notice that the metric $g$ on the universal cover is almost never complete, thus Theorem~\ref{thmLocallyMetric} is not a consequence of Theorem~\ref{thmDeRham}.

\subsection{Transverse similarity structures} \label{sectIntroFoliations}

The main tool in the proofs of Theorems~\ref{thmLocallyMetric} and~\ref{thmConformalStructure} is the study of transverse similarity structures on foliations. Such foliations may be seen as a particular case of (transversally) conformal foliations, or a generalization of (transversally) Riemannian foliations. The precise definition is given at the beginning of Section~\ref{sectFoliations}.

Our main result on transverse similarity structures is the following:
\begin{thm} \label{thmFoliations}
Let $(M, \mathcal F)$ be a compact foliated manifold endowed with a transverse similarity structure. Then one of the following two facts occurs:
\begin{enumerate}
\item The transverse similarity structure on the foliation $\mathcal F$ is flat (\emph{i.e.} the metric $g$ on the transversal $T$ is flat),
\item The foliation $\mathcal F$ is transversally Riemannian (\emph{i.e.} there exists a metric $h$ on the transversal $T$ such that the transition maps are isometries).
\end{enumerate}
\end{thm}

We prove Theorem~\ref{thmFoliations} in Section~\ref{sectFoliations}. Notice that we do not assume that the transverse similarity structure on the foliation is induced by a locally metric connection on $M$.

Foliations endowed with a \emph{flat} transverse similarity structure (\emph{i.e.} those which correspond to the first case of Theorem~\ref{thmFoliations}) were completely classified by Ghys~\cite{MR1125840} when $M$ has dimension $3$ and $\mathcal F$ has dimension $1$. See also~\cite{MR1191025} and~\cite{MR1611816} on this subject.

\bigskip \noindent \textbf{About the foliated Ferrand-Obata conjecture.} For transversally \emph{conformal} foliations, there is an analogue of Theorem~\ref{thmFoliations} (see \cite{MR2073842}):
\begin{thm}[Tarquini, 2004] \label{thmTarquini}
Any transversally analytic conformal foliation of codimension $\geq 3$, on a compact connected manifold, is either transversally Möbius or Riemannian.
\end{thm}

It is also believed that Theorem~\ref{thmTarquini} should be valid without the analyticity assumption: this is the \emph{foliated Ferrand-Obata conjecture}. Our Theorem~\ref{thmFoliations} implies the following:
\begin{cor}
The foliated Ferrand-Obata conjecture is true in the special case where the transverse conformal structure on the foliation is induced by a transverse similarity structure.
\end{cor}

\subsection{Structure of the paper}
We start by proving Theorem~\ref{thmFoliations} in Section~\ref{sectFoliations}. We use Theorem~\ref{thmFoliations} to prove Theorems~\ref{thmConformalStructure} and~\ref{thmLocallyMetric} in Section~\ref{sectEndProofs}. Then, we show Theorem~\ref{thmClosures} in Section~\ref{sectTori}, and use Theorem~\ref{thmClosures} to prove Theorem~\ref{thmDim2} in Section~\ref{sectDim2}.

\section{Foliations endowed with transverse similarity structures} \label{sectFoliations}

In this section, we prove Theorem~\ref{thmFoliations}.

If $(M,g)$ is a Riemannian manifold, its similarity pseudogroup $\mathrm{Sim}_\mathrm{loc}(M)$ consists of all $\phi : U \to V$ such that $\phi^*g = \lambda^2 g$, where $U$ and $V$ are open subsets of $M$, and $\lambda \in \mathbb R_{>0}$ is locally constant on $U$. For any $x \in M$, the number $\lambda(x)$ is called the \emph{ratio of $\phi$ at $x$} (if $M$ is connected, there is no need to specify the point $x$).

A foliation of a compact manifold $(M, \mathcal F)$ is a covering $(U_i)_{1 \leq i \leq r}$ with the following structure:
\begin{enumerate}
\item For each $i \in \{1, \ldots, r \}$, $U_i$ is diffeomorphic to $V_i \times T_i$, where $V_i$ is an open ball of $\mathbb R^p$ and $T_i$ an open ball of $\mathbb R^q$. This gives us natural projections $f_i: U_i \to T_i$. The disjoint union $T = \bigcup_i T_i$ is called the \emph{global transversal}.
\item There exist \emph{transition maps} which are diffeomorphisms $(\gamma_{ij})_{i,j} : f_i(U_i \cap U_j) \to f_j(U_i \cap U_j)$ such that $f_j = \gamma_{ij} \circ f_i$ on $U_i \cap U_j$.
\end{enumerate}

The pseudogroup $\Gamma$ spanned by the $(\gamma_{ij})$ is called the \emph{holonomy pseudogroup} of the foliation.

\begin{remarque}
In this paper, we use the notions of ``holonomy group'', from the theory of Riemannian manifolds, and ``holonomy pseudogroup'', from the theory of foliations: these two notions must not be confused.
\end{remarque}

A transverse similarity structure on the foliation $\mathcal F$ is a metric $g$ on the transversal $T$ such that the transition maps $\gamma_{ij}$ are local similarities (\emph{i.e.} belong to $\mathrm{Sim}_\mathrm{loc}(T)$). The foliation is said to be \emph{transversally Riemannian} (or simply \emph{Riemannian}) if it is possible to choose $g$ such that the $\gamma_{ij}$ are isometries.

A foliation is said to be \emph{equicontinuous} if there exists a Riemannian metric on the transversal such that its holonomy pseudogroup $\Gamma$ is equicontinuous. If the foliation has a transverse similarity structure, equicontinuity is equivalent to the existence of a constant $m > 1$ such that the ratio of any $\gamma \in \Gamma$ at any $x \in M$ lies in the interval $[1/m, m]$.

The following proposition, which is proved in~\cite{MR2078076}, is crucial in the proof of Theorem~\ref{thmFoliations}:

\begin{prop} \label{propEquicontinuous}
Any equicontinuous foliation endowed with a transverse similarity structure is Riemannian.
\end{prop}

Now, our first step in the proof is based on a trick which was described in~\cite{MR1889249}.

\begin{prop} \label{pseudodistance}
Let $(M, g)$ be a Riemannian manifold whose Riemann tensor $R$ does not vanish anywhere (\emph{i.e.} there is no $x \in M$ such that $R_x = 0$). Then $\mathrm{Sim}_\mathrm{loc}(M)$ preserves a smooth Riemannian metric.
\end{prop}
\begin{proof}
If $R$ denotes the Riemann tensor, define $\norm{R}_g(x)$ as the supremum of the values of $\norm{R_x(u, v) w}_g$ when $u, v, w$ are vectors of $T_xM$ which have unit length for $g$. Notice that $\norm{R}_g(x)$ is finite for each $x \in M$, because the unit sphere of $T_xM$ for the metric $g$ is compact. Then the metric $\norm{R}_g g$ is invariant by $\mathrm{Sim}_\mathrm{loc}(M)$.
\end{proof} (

Thus, if $(M, \mathcal F)$ is a foliated manifold endowed with a transverse similarity structure, either $\mathcal F$ is Riemannian, or the Riemann tensor of $(T, g)$ vanishes somewhere. Our aim is to show that, in the last case, the Riemann tensor vanishes in fact everywhere.

Until the end of this section, we consider a \emph{compact, connected} foliated manifold $(M, \mathcal F)$ endowed with a transverse similarity structure. We consider a covering of $M$ by open sets $U_i$ which are diffeomorphic to $V_i \times T_i$, the projections $f_i$, the transversal $T$, the transition maps $\gamma_{ij}$, the holonomy pseudogroup $\Gamma$, and the metric $g$ on the transversal. This metric $g$ induces a distance $d_i$ on each $T_i$.

Intuitively, the holonomy pseudogroup may be defined as the set of local diffeomorphisms of the transversal obtained by ``sliding along the leaves''.
\begin{definition}
A piecewise $C^1$ path $c: [a,b] \to M$ is \emph{vertical} if for all $t_0 \in [a,b]$, for all $i \in \{1, \ldots, r\}$ such that $c(t_0) \in U_i$, the mapping $t \mapsto f_i(c(t))$ is constant in a neighborhood of $t_0$.
\end{definition}

The \emph{leaf} which contains $x \in M$ is defined as the set of all possible $c(t)$, where $c$ is a piecewise $C^1$ vertical path such that $c(0) = x$.

\begin{definition} \label{defHolonomyPath}
Consider a piecewise $C^1$ vertical path $c: [a,b] \to M$, and $i, j \in \{1, \ldots, r\}$ such that $c(a) \in U_i$, $c(b) \in U_j$. Define $x = f_i(c(a))$.

Choose a sequence of times $a = t_1 < \ldots < t_{p+1} = b$ and a sequence of indices $i_1, \ldots, i_p$, such that for all $l \in \{1, \ldots, p\}$ and all $t \in [t_l, t_{l+1}]$,
\[ c(t) \in U_{i_l}. \]

The \emph{holonomy germ} $\gamma$ from $T_i$ to $T_j$ at $x$ obtained by sliding along $c$ is defined as the germ at $x$ of the diffeomorphism \[ \gamma_{i_p j} \circ \gamma_{i_{p-1} i_p} \circ \dots \circ \gamma_{i_1 i_2} \circ \gamma_{i i_1}. \]
\end{definition}

The following two propositions are basic properties of holonomy pseudogroups: see for example Chapter~1 of~\cite{MR932463} for details.

\begin{prop} \label{lemmaSliding}
The \emph{holonomy germ} is well-defined: it depends only on the path $c$ and the choice of $i$ and $j$. In particular, it does not depend on the choice of a sequence of times $t_1, \ldots, t_{p+1}$ or the choice of a sequence of indices $i_1, \ldots, i_p$.
\end{prop}

\begin{lemme} \label{lemmaExistsPath}
Consider an element of the holonomy pseudogroup
\[ \gamma = \gamma_{i_{p-1} i_p} \circ \dots \circ \gamma_{i_1 i_2} \] and an element $x \in T_{i_1}$.
Then there exists a piecewise $C^1$ path $c$ in $M$ such that the germ of $\gamma$ at $x$ is the holonomy germ from $T_{i_1}$ to $T_{i_p}$ at $x$ obtained by sliding along $c$.
\end{lemme}
\begin{proof}
We construct $c : [1, p+1] \to M$ such that for each $l \in \{1, \ldots, p\}$, and all $t \in [l, l+1]$:
\[ f_{i_l}(c(t)) = \gamma_{i_{l-1} i_l} \circ \dots \circ \gamma_{i_1 i_2} (x). \]
This is possible because $V_i$ is path-connected for each $i$.
\end{proof}

\begin{lemme} \label{epsilon}
There exists $\epsilon_0 > 0$ such that for all $x \in M$, there exists $i \in \{1, \ldots, r\}$ which satisfies $x \in U_i$ and $d_i(f_i(x), \partial T_i) > \epsilon_0$ (see the beginning of Section~\ref{sectFoliations} for the notations).
\end{lemme}
\begin{proof}
Assume the contrary: there exists a sequence $(x_n)_{n \in \mathbb N}$ in $M$ such that for all $i \in \{1, \dots, r\}$ with $x_n \in U_i$, we have $d_i(f_i(x_n), \partial T_i) \leq 1/n$. Since $M$ is closed, we may assume that $x_n$ converges to some $x_\infty \in M$. Then $x_\infty$ is in some $U_{i_0}$, and for any large enough $n$, $x_n \in U_{i_0}$. Hence, $d_{i_0}(f_{i_0}(x_n), \partial T_{i_0}) \to 0$, which contradicts the fact that $x_\infty \in U_{i_0}$.
\end{proof}

In the following, we fix this $\epsilon_0$.

\begin{definition} \label{defNotation}
Let $x \in M$, $p \in \mathbb N$ and $i_1, \ldots, i_p \in \{1, \ldots, r\}$. We will write \[ \gamma = \gamma_{i_{p-1} i_p} \underset{x}{\circledcirc} \dots \underset{x}{\circledcirc} \gamma_{i_1 i_2} \] if:
\begin{enumerate}
\item $\gamma = \gamma_{i_{p-1} i_p} \circ \dots \circ \gamma_{i_1 i_2}$,
\item For all $l \in \{1, \ldots, p-1\}$, $d_{i_l}(\gamma_{i_{l-1} i_l} \circ \dots \circ \gamma_{i_1 i_2} (f_{i_1}(x)), \partial U_{i_l}) > \epsilon_0$,
\item For all $l \in \{1, \ldots, p-1\}$, the domain of $\gamma_{i_l i_{l+1}}$ contains the ball $B_g(\gamma_{i_{l-1} i_l} \circ \dots \circ \gamma_{i_1 i_2} (f_{i_1}(x)), \epsilon_0)$.
\end{enumerate}
Here, $B_g(y, r)$ denotes the ball of center $y$ and radius $r$ for the metric $g$.
\end{definition}

Recall that we have defined the \emph{ratio} of a local similarity at the beginning of Section~\ref{sectFoliations}.

\begin{remarque}
When $p = 1$, we will use the convention that $\gamma_{i_{p-1} i_p} \circ \dots \circ \gamma_{i_1 i_2} = \mathrm{Id}$. 
\end{remarque}

\begin{lemme} \label{lemmaExtensionDomain}
Consider $x \in M$ and an element $\gamma = \gamma_{i_{p-1} i_p} \underset{x}{\circledcirc} \dots \underset{x}{\circledcirc} \gamma_{i_1 i_2}$ of the holonomy pseudogroup.

For each $l \in \{1, \ldots, p\}$, write $r_l$ the ratio of $\gamma_{i_{l-1} i_l} \circ \dots \circ \gamma_{i_1 i_2}$ at $f_{i_1}(x)$ (in particular $r_1 = 1$).

Then the domain of $\gamma$ contains \[ B_g\left(f_{i_1}(x), \frac{\epsilon_0}{\underset{1 \leq l \leq p-1}{\max} r_l}\right). \]
\end{lemme}
\begin{proof}
We prove the lemma by induction on $p$. For $p = 1$ this results from Definition~\ref{defNotation}.

Consider an element $\gamma = \gamma_{i_{p-1} i_p} \underset{x}{\circledcirc} \dots \underset{x}{\circledcirc} \gamma_{i_1 i_2}$ of the holonomy pseudogroup. Assume (induction hypothesis) that the domain of $\tilde \gamma = \gamma_{i_{p-2} i_{p-1}} \underset{x}{\circledcirc} \dots \underset{x}{\circledcirc} \gamma_{i_1 i_2}$ contains the ball $B_g\left(f_{i_1}(x), \frac{\epsilon_0}{\max_{1 \leq l \leq p-2} r_l}\right)$. Then \[ \tilde\gamma \left(B_g\left(f_{i_1}(x), \frac{\epsilon_0}{\underset{1 \leq l \leq p-1}{\max} r_l}\right)\right) = B_g\left(\tilde \gamma (f_{i_1}(x)), \frac{r_{p-1} \epsilon_0}{\underset{1 \leq l \leq p-1}{\max} r_l}\right). \] Moreover, the domain of $\gamma_{i_{p-1} i_p}$ contains $B_g(\tilde \gamma (f_{i_1}(x)), \epsilon_0)$ (by Definition~\ref{defNotation}), which itself contains $B_g\left(\tilde \gamma (f_{i_1}(x)), \frac{r_{p-1} \epsilon_0}{\max_{1 \leq l \leq p-1} r_l}\right)$.

Thus the domain of $\gamma = \gamma_{i_{p-1}i_p} \circ \tilde \gamma$ contains the ball $B_g\left(f_{i_1}(x), \frac{\epsilon_0}{\max_{1 \leq l \leq p-1} r_l}\right)$.
\end{proof}

\begin{lemme} \label{lemmaDefinitionGamma}
Let $\gamma \in \Gamma$, $x \in M$ and $i \in \{1, \ldots, r\}$, such that $\gamma$ is defined on a neighborhood of $f_i(x)$ in $T_i$ and takes its values in $T_j$.

Then there exists $\tilde \gamma = \gamma_{i_{p-1} i_p} \underset{x}{\circledcirc} \dots \underset{x}{\circledcirc} \gamma_{i_1 i_2}$ defined on a neighborhood of $f_{i_1}(x)$, which has the same germ as $\gamma_{ji_p} \circ \gamma \circ \gamma_{i_1i}$ at $f_{i_1}(x)$.
\end{lemme}
\begin{proof}
It results from Lemma~\ref{lemmaExistsPath} that the germ of $\gamma$ at $f_i(x)$ is the holonomy germ from $T_i$ to $T_j$ at $x$ obtained by sliding along a curve $c : [a, b] \to M$, such that $f_i(c(a)) = f_i(x)$, and $f_j(c(b)) = \gamma(f_i(x))$. For each $l \in \{1, \dots, r\}$, we define $E_l$ as the set of all open subsets $\mathcal I$ of $[a,b]$ which are intervals such that for all $t \in \mathcal I$, $c(t) \in U_l$ and $d_l(f_l(c(t)), \partial T_l) > \epsilon_0$. Lemma~\ref{epsilon} implies that $\cup_{1 \leq l \leq r} E_l$ is an open cover of $[a,b]$: it has a finite subcover $\left\{[a_1, b_1), (a_2, b_2), \ldots, (a_{p-1}, b_{p-1}), (a_p, b_p]\right\}$, where $a_1 = a$ and $b_p = b$. We may assume that for all $k \in \{1, \ldots, p-1\}$, $a_k \leq a_{k+1} \leq b_k \leq b_{k+1}$. For all $k \in \{1, \ldots, p \}$, choose $i_k \in \{ 1, \ldots, r \}$ such that $(a_k, b_k) \in E_{i_k}$, and define $\tilde \gamma = \gamma_{i_{p-1} i_p} \underset{x}{\circledcirc} \dots \underset{x}{\circledcirc} \gamma_{i_1 i_2}$. Then, by Lemma~\ref{lemmaSliding}, $\tilde \gamma$ has the same germ as $\gamma_{ji_p} \circ \gamma \circ \gamma_{i_1i}$ at $f_{i_1}(x)$.
\end{proof}

\begin{remarque}
The arguments developed in Lemmas~\ref{lemmaExtensionDomain} and~\ref{lemmaDefinitionGamma} are closely related to the notion of \emph{compactly generated pseudogroups} (see for example~\cite{MR1125840} for a definition).
\end{remarque}

\begin{lemme} \label{lemmaFoliations}
Let $E$ be the set of all $x \in M$ for which there exists $m > 1$ such that for all $i \in \{1, \dots, r\}$ with $x \in U_i$, every $\gamma \in \Gamma$ defined on $f_i(x)$ has ratio $\geq 1/m$ at $f_i(x)$.
\begin{enumerate}
\item In the definition of $E$, it is possible to choose $m$ independently of $x$.
\item If $E$ is non-empty, then $E = M$ and $\Gamma$ is equicontinuous.
\end{enumerate}
\end{lemme}

\begin{proof}
We start with the proof of the first statement. Assume that there is no uniform bound. Then, there exist sequences $(x^n)$, $(i^n)$, $(j^n)$ and $(\gamma^n)$ such that $x^n \in E$, $\gamma^n$ is defined on a neighborhood of $f_{i^n}(x^n)$ in $T_{i^n}$, takes its values in $T_{j^n}$, and the ratio of $\gamma^n$ is $\leq 1/n$ at $f_{i^n}(x^n)$ (the upper indices do not indicate exponentiation). Let $r_\mathrm{max}$ be the maximum ratio of $\gamma_{ij}$ for $i,j \in \{ 1, \ldots, r \}$.

For each $n$, Lemma~\ref{lemmaDefinitionGamma} gives us a $\tilde \gamma^n = \gamma_{i_{{p^n}-1}^n i_{p^n}^n} \underset{x^n}{\circledcirc} \dots \underset{x^n}{\circledcirc} \gamma_{i_1^n i_2^n}$, which has the same germ as $\gamma_{j^ni_{p^n}^n} \circ \gamma^n \circ \gamma_{i_1^ni^n}$ at $f_{i_1^n}(x^n)$. Notice that $\tilde \gamma^n$ has ratio $\leq r_\mathrm{max}^2/n$ at $f_{i_1^n}(x^n)$ (because $\tilde \gamma^n$ coincides with $\gamma_{j^ni_{p^n}^n} \circ \gamma^n \circ \gamma_{i_1^ni^n}$ on a neighborhood of $f_{i_1^n}(x^n)$, and the ratio of $\gamma^n$ is $\leq 1/n$ at $f_{i^n}(x^n)$).

Choose ${q^n} \in \{1, \ldots, {p^n}\}$ which minimizes the ratio of $\gamma_{i_{{p^n}-1}^n i_{p^n}^n} \circ \dots \circ \gamma_{i_{q^n}^n i_{{q^n}+1}^n}$ at $\gamma_{i_{{q^n}-1}^n i_{q^n}^n} \circ \dots \circ \gamma_{i_1^n i_2^n} (f_{i_1^n}(x^n))$ (in particular this ratio is $\leq r_\mathrm{max}^2/n$), and write $\tilde \rho^n = \gamma_{i_{{p^n}-1}^n i_{p^n}^n} \circ \dots \circ \gamma_{i_{q^n}^n i_{{q^n}+1}^n}$. Choose $y^n$ such that $f_{i_{q^n}^n}(y^n) = \gamma_{i_{{q^n}-1}^n i_{q^n}^n} \circ \dots \circ \gamma_{i_1 i_2} (f_{i_1^n}(x^n))$. Notice that $y^n \in E$.

By Lemma~\ref{lemmaExtensionDomain}, $\tilde \rho^n$ is well-defined on $B_g(f_{i_{q^n}^n}(y^n), \epsilon_0)$ and has ratio $\leq r_\mathrm{max}^2/n$ at $f_{i_{q^n}^n}(y^n)$.

Since $M$ is compact, we may assume up to extraction that $(y^n)$ converges to a limit $y \in M$ (and $y \in U_i$ for some $i$): there exists $n_0 > 0$ such that for all $n \geq n_0$, $y^n \in U_i$ and $d_i(f_i(y^n), f_i(y)) < \epsilon_0/(3 r_\mathrm{max})$. Thus, $\tilde \rho^{n}$ is well-defined on $f_{i_{q^{n_0}}^{n_0}}(y^{n_0})$ for all $n \geq n_0$, which contradicts the fact that $y^{n_0} \in E$ and ends the proof of the first statement.

\bigskip

To prove the second statement, first notice that for all $x \in E$, and all $i \in \{1, \dots, r\}$ such that $x \in U_i$, every $\gamma \in \Gamma$ defined on $f_i(x)$ (taking values in $T_j$) has ratio $\leq m$ at $f_i(x)$: otherwise, $\gamma^{-1}$ would have ratio $< 1/m$ at $f_j(x)$, which contradicts the fact that $\gamma(f_i(x)) \in f_j(E \cap U_j)$. 

Since $M$ is connected, it suffices to show that $E$ is open and closed in $M$. Thus, $\Gamma$ will be equicontinuous on $M$.

Let us show that $E$ is open. Let $x_0 \in E$ and $i_1$ such that $d_{i_1}(f_{i_1}(x_0), \partial T_{i_1}) > \epsilon_0$. Consider $V$ a neighborhood of $x_0$ such that $V \subseteq U_{i_1}$ and $f_{i_1}(V) \subseteq B_g(f_{i_1}(x_0), \epsilon_0/(2 m))$. Let us show that $V \subseteq E$: let $y_0 \in V$, $i \in \{1, \dots, r\}$, and $\gamma \in \Gamma$ defined on a neighborhood of $f_i(y_0)$, taking its values in $T_j$.

With Lemma~\ref{lemmaDefinitionGamma}, there exists a $\tilde \gamma = \gamma_{i_{p-1} i_p} \underset{y_0}{\circledcirc} \dots \underset{y_0}{\circledcirc} \gamma_{i_1 i_2}$, which has the same germ as $\gamma_{ji_p} \circ \gamma \circ \gamma_{i_1i}$ at $f_{i_1}(y_0)$.


Let us prove, by induction on $l$, that for all $l \in \{1, \ldots, p\}$, the ratio of $\gamma_{i_{l-1} i_l} \circ \dots \circ \gamma_{i_1 i_2}$ at $f_{i_1}(y_0)$ is between $1/m$ and $m$. For $l = 1$, there is nothing to prove. For $l = 2$, $\gamma_{i_1 i_2}$ is well-defined on $B_g(f_{i_1}(y_0), \epsilon_0)$, which contains $f_{i_1}(x_0)$. Since $x_0 \in E$, this implies that the ratio of $\gamma_{i_1 i_2}$ at $f_{i_1}(y_0)$ is between $1/m$ and $m$. We now assume that the assertion is true for all $l \leq l_0$, where $l_0 \in \{1, \ldots, p-1\}$. Then by Lemma~\ref{lemmaExtensionDomain}, $\gamma_{i_{l_0} i_{l_0+1}} \circ \dots \circ \gamma_{i_1 i_2}$ is well-defined on $B_g(f_{i_1}(y_0), \epsilon_0 / m)$, which contains $f_{i_1}(x_0)$. The fact that $x_0 \in E$ now implies that the ratio of $\gamma_{i_{l_0} i_{l_0+1}} \circ \dots \circ \gamma_{i_1 i_2}$ at $f_{i_1}(y_0)$ is between $m$ and $1/m$, which concludes the induction.

Therefore, the ratio of $\tilde \gamma$ is between $1/m$ and $m$ at $f_{i_1}(x_0)$. The ratio of $\gamma$ is at least $1/(r_\mathrm{max}^2 m)$ at $f_i(y_0)$, so $y_0 \in E$, and $E$ is open.

Now, we show that $M \setminus E$ is open in $M$. Let $x_0 \in M \setminus E$, $i \in \{1, \dots, r\}$, and $\gamma \in \Gamma$ defined on $f_i(x_0)$ with ratio $< 1/m$. Then $\gamma$ is defined on a connected open set $W \subseteq T_i$ containing $f_i(x_0)$, and $f_i^{-1}(W)$ is an open set of $M$, containing $x_0$ and contained in $M \setminus E$, so $M \setminus E$ is open.
\end{proof}

\bigskip \noindent {\bf End of the proof of Theorem~\ref{thmFoliations}.} Assume that $(T, g)$ is not flat, and let $T'$ be the set of all $y \in T$ at which the Riemann tensor of $g$ is nonzero. Notice that $T'$ is stable under the holonomy pseudogroup $\Gamma$. Now, Proposition~\ref{pseudodistance} gives us a Riemannian metric $g'$ on $T'$ which is invariant by $\mathrm{Sim}_\mathrm{loc}(T')$, and therefore invariant by the holonomy pseudogroup $\Gamma$. Hence, the set $E$ defined in Lemma~\ref{lemmaFoliations} is non-empty. By Lemma~\ref{lemmaFoliations}, $\Gamma$ is equicontinuous. Finally, in view of Proposition~\ref{propEquicontinuous}, $\mathcal F$ is a Riemannian foliation, and Theorem~\ref{thmFoliations} is proved.

\section{Decomposition theorems for locally metric connections} \label{sectEndProofs}

In this section, we prove Theorems~\ref{thmLocallyMetric} and~\ref{thmConformalStructure}. The group of affine mappings of a Riemannian manifold $(M,g)$ is the group of all diffeomorphisms of $M$ which preserve the Levi-Civita connection of $g$.

The proof relies on the following theorem:

\begin{thm}[Ponge-Reckziegel, 1993] \label{thmPonge}
Let $M$ be a simply connected Riemannian manifold, whose Levi-Civita connection $\nabla$ is reducible: thus, the tangent bundle $TM$ admits two complementary orthogonal distributions $E'$ and $E''$ invariant by parallel transport, which determine foliations $\mathcal F'$ and $\mathcal F''$. Assume that the leaves of $\mathcal F'$ are all complete. Then, $M$ is globally isometric to a product of Riemannian manifolds $M' \times M''$, and the foliations $\mathcal F'$ and $\mathcal F''$ are determined by the product structure.
\end{thm}

Theorem~\ref{thmPonge} is a particular case of the main result of~\cite{MR1245571}. In fact, the classical proof of the de Rham theorem given in~\cite{MR0152974} also adapts directly to this case with very few changes.

We will also need the following theorem (for a proof, see for example~\cite{MR0152974}, page 185):
\begin{thm} \label{thmLocalDeRham}
Consider a simply connected Riemannian manifold $(\tilde M, \tilde g)$ and a point $\tilde x \in \tilde M$. Let $E_{\tilde x}^0$ be the maximal linear subspace of the tangent space $T_{\tilde x} \tilde M$ on which $\mathrm{Hol}_{\tilde x}(\tilde \nabla)$ acts trivially, and define $E_{\tilde x}^{>0}$ as the orthogonal complement of $E_{\tilde x}^0$ in $T_{\tilde x} \tilde M$. Then there exists a decomposition of $E_{\tilde x}^{>0}$, unique up to the order of the factors, into mutually orthogonal, irreducible subspaces, which are invariant by $\mathrm{Hol}_{\tilde x}(\tilde \nabla)$:
\[ E_{\tilde x}^{>0} = E_{\tilde x}^1 \oplus \ldots \oplus E_{\tilde x}^k. \]
\end{thm}

The following lemma is classical:
\begin{lemme} \label{lemmaAffSim}
Consider a connected Riemannian manifold $(M, g)$ with its Levi-Civita connection $\nabla$. If $\nabla$ has irreducible holonomy, then:
\begin{enumerate}
	\item the only metrics whose Levi-Civita connection is $\nabla$ are the metrics $h_\lambda = \lambda^2 g$, $\lambda > 0$,
	\item $\mathrm{Aff}(M, g) = \mathrm{Sim}(M, g)$.
\end{enumerate}
\end{lemme}
\begin{proof}
\begin{enumerate}
\item Let $h$ be a metric whose Levi-Civita connection is $\nabla$ and let $x \in M$. Define a linear mapping $F_x : T_xM \to T_xM$ in the following way: for all $u \in T_xM$, $F_x(u)$ is the unique vector such that $g_x(u, \cdot) = h_x(F_x(u), \cdot)$. Since $\mathrm{Hol}_x(\nabla)$ preserves $g_x$ and $h_x$, the eigenspaces of $F_x$ are invariant under $\mathrm{Hol}_x(\nabla)$. Since $\nabla$ is irreducible, the only possible eigenspaces for $F_x$ are $\{0\}$ and $T_xM$. But $F_x$ is self-adjoint (for both metrics $g$ and $h$), so $F_x$ is a homothety. This shows that there exists $\lambda_x > 0$ such that $h_x = \lambda_x^2 g_x$.

Now, we prove that $\lambda_x$ does not depend on $x$: for $x, y \in M$, choose any nonzero vector $u \in T_xM$ and any path $c : [0,1] \to M$ joining $x$ to $y$: if $v$ is obtained by the parallel transport of $u$ along $c$, we have $h_y(v) = h_x(u) = \lambda_x^2 g_x(u) = \lambda_x^2 g_y(v)$, and therefore $\lambda_x = \lambda_y$ (since $g$ and $h$ have the same Levi-Civita connection, the parallel transport is the same for $g$ and $h$).
\item For all $\phi \in \mathrm{Aff}(M,g)$, the metric $\phi^*g$ is preserved by the Levi-Civita connection $\nabla$ of $g$, so $\phi^*g$ is proportional to $g$ and therefore $\phi$ is a similarity.
\end{enumerate}
\end{proof}

We are now ready to prove Theorem~\ref{thmLocallyMetric}. Consider a compact connected manifold $(M, \nabla)$, where $\nabla$ is locally metric, and its universal cover $(\tilde M, \tilde \nabla)$, on which there is a metric $\tilde g$ preserved by $\tilde \nabla$. Fix $x \in M$ and choose a preimage $\tilde x \in \tilde M$.

Consider the subspaces $E_{\tilde x}^0$ and $E_{\tilde x}^{>0} = E_{\tilde x}^1 \oplus \ldots \oplus E_{\tilde x}^k$ given by Theorem~\ref{thmLocalDeRham}. 

This induces a decomposition $T_x M = E_x^0 \oplus E_x^{>0}$. Furthermore, since $\pi_1(M)$ acts on $\tilde M$ by preserving the connection $\tilde \nabla$, this decomposition does not depend on the choice of the preimage $\tilde x$ of $x$, up to the order of the factors. Thus, the holonomy group $\mathrm{Hol}_x(\nabla)$ acts on $E_x^{>0}$ by permuting the factors: by considering a finite cover of $M$, one may assume that $\mathrm{Hol}_x(\nabla)$ preserves the decomposition of $T_x M$. Then, one may consider $E'$ the distribution on $\tilde M$ obtained by parallel transport of $E_x^k$, and $E''$ obtained by parallel transport of $E_x^0 \oplus \ldots \oplus E_x^{k-1}$. If $E'' = \{0\}$, then the manifold $\tilde M$ has irreducible holonomy and therefore the conclusion of Theorem~\ref{thmLocallyMetric} is satisfied. We will now assume that $E'' \neq \{0\}$. Since the distributions $E'$ and $E''$ are invariant by parallel transport, they are integrable (see~\cite{MR0152974}, page 180), so that they induce transverse foliations $\mathcal F'$ and $\mathcal F''$ on $M$. By pullback, one obtains distributions $\tilde{E'}$ and $\tilde{E''}$ on $\tilde M$, and foliations $\tilde{\mathcal F'}$ and $\tilde{\mathcal F''}$ on $\tilde M$.

The local version of De Rham's theorem gives us a covering $(U_i)_{1 \leq i \leq r}$ of $M$ compatible with the foliations $\mathcal F'$ and $\mathcal F''$, such that each $U_i$ is diffeomorphic to $V_i \times T_i$, where $V_i$ (the plaque of $\mathcal F''$) is an open ball of $\mathbb R^p$ and $T_i$ (the plaque of $\mathcal F'$) an open ball of $\mathbb R^q$. Write the projections $f_i: U_i \to T_i$. The connection $\nabla$ on $M$ induces a connection $\nabla_T$ on the transversal $T = \cup_{1 \leq i \leq r} T_i$, which is preserved by the holonomy pseudogroup of $\mathcal F''$. Since each component of $T$ is simply connected, $\nabla_T$ is the Levi-Civita connection of a Riemannian metric $g_T$ on $T$. The holonomy pseudogroup of $\mathcal F''$ acts by affine transformations on $(T, g_T)$. But since $\nabla_T$ is irreducible, these transformations are in fact local similarities of $(T,g_T)$ (by Lemma~\ref{lemmaAffSim}), which implies that $\mathcal F''$ has a transverse similarity structure. 

By construction, the holonomy group of $M$ does not act trivially on $E'$, so $\mathcal F''$ is not transversally flat. With Theorem~\ref{thmFoliations}, $\mathcal F''$ can be equipped with a transverse Riemannian structure: we obtain a new Riemannian metric $h_T$ on the transversal $T$ such that the holonomy pseudogroup of $\mathcal F''$ acts by isometries on $(T, h_T)$.

Consider any Riemannian metric $\tilde g$ on the universal cover $\tilde M$. We are going to construct a new Riemannian metric $h$ on $\tilde M$. Choose $\tilde x \in \tilde M$ and two vectors $v, w \in T_{\tilde x} \tilde M$, and consider their projections $x \in M$ and $v, w \in T_x M$. Define $h_{\tilde x}(\tilde v, \tilde w)$ in the following way:
\begin{enumerate}
\item If $\tilde v, \tilde w \in \tilde E_{\tilde x}''$, then $h_{\tilde x}(\tilde v , \tilde w) = \tilde g_{\tilde x}(\tilde v, \tilde w)$,
\item If $\tilde v, \tilde w \in \tilde E_{\tilde x}'$, then $h_{\tilde x}(\tilde v , \tilde w) = ((f_i)^*(h_T))_x(v, w)$ for any $i$ such that $x \in U_i$ (since the holonomy pseudogroup of $\mathcal F''$ acts by isometries, the result does not depend on the choice of $i$),
\item If $\tilde v \in \tilde E_{\tilde x}'$ and $\tilde w \in \tilde E_{\tilde x}''$, then $h_{\tilde x}(\tilde v , \tilde w) = 0$.
\end{enumerate}

By construction, the metric $h$ is locally a product of metrics, so that the distributions $\tilde{E'}$ and $\tilde{E''}$ are invariant by parallel transport with respect to $h$. Furthermore:

\begin{prop}
The metric induced by $h$ on the leaves of $\tilde {\mathcal F'}$ is complete.
\end{prop}
\begin{proof}
By Lemma~\ref{epsilon}, there exists $\epsilon_0 > 0$ such that for each $\tilde x \in \tilde M$, there exists $i \in \{1, \ldots, r\}$ such that the projection $x \in M$ of $\tilde x$ satisfies $x \in U_i$ and $d_i(f_i(x), \partial T_i) > \epsilon_0$, where $d_i$ is the distance induced by the metric $h$ on $T_i$. This implies that for every $\tilde x \in \tilde M$, in the leaf of $\tilde{\mathcal F'}$ which contains $\tilde x$, the ball of center $\tilde x$ and radius $\epsilon_0$ for the metric $h$ is compact. Thus, the metric on the leaves of $\tilde{\mathcal F'}$ is complete.
\end{proof}

Therefore, by Theorem~\ref{thmPonge}, $\tilde M$ is globally the product of two Riemannian manifolds $\tilde M'$ and $\tilde M''$ whose tangent distributions are $\tilde{E'}$ and $\tilde{E''}$. The existence of the de Rham decomposition follows by induction on the dimension of $M$. Thus, Theorem~\ref{thmLocallyMetric} is proved.

For the proof of Theorem~\ref{thmConformalStructure}, we will need the following propositions:

\begin{prop} \label{propHomothetyFlat}
Consider a connected Riemannian manifold $(M, g)$, and a similarity $\phi \in \mathrm{Sim}(M)$. Assume that $\phi$ has a fixed point $x \in M$, and that its ratio is $r_\phi < 1$. Then the manifold $M$ is isometric to the Euclidean $\mathbb R^q$ for some $q \geq 0$.
\end{prop}
\begin{proof}
First, let us prove that $M$ is flat. Choose any $y \in M$ and four vectors $a, b, c, d$ in $T_yM$ of unit length for $g$. The point is that $\phi$ preserves $R$, \emph{i.e.}
\[ R (\phi_* a, \phi_* b) \phi_* c = \phi_* R (a, b) c. \]
Thus:
\[ \begin{aligned}
 \prodscal{R (a, b) c}{d} & = r_\phi^{-2n} \prodscal{\phi_*^n R (a, b) c}{\phi_*^n d}
 \\ &= r_\phi^{-2n} \prodscal{R (\phi_*^n a, \phi_*^n b) \phi_*^n c}{\phi_*^n d}
 \\ & \leq r_\phi^{-2n} r_\phi^{4n} \norm{R}_g (\phi^n(y)).
 \end{aligned} \]

Since $\phi^n(y)$ tends to the fixed point $x$, the quantity $\norm{R}_g (\phi^n(y))$ is bounded. Thus, $\prodscal{R (a, b) c}{d} = 0$, and therefore, $M$ is flat.

\bigskip

Now, since $M$ is flat, the exponential map $\exp_x : B(0,\epsilon) \to B_g(x, \epsilon)$ is an isometry for some $\epsilon > 0$ (where $B(0,\epsilon)$ is the ball in $T_xM$ of center $0$ and radius $\epsilon$ for the Euclidean metric $g_x$, while $B_g(x, \epsilon)$ is the ball in $M$ of center $x$ and radius $\epsilon$ for the distance induced by $g$).

Thus, for all $n \geq 0$, $\phi^{-n} \circ \exp_x \circ D_x \phi^n$ is an isometry from $B(0, r_\phi^{-n} \epsilon)$ to $B_g(x, r_\phi^{-n} \epsilon)$. Since $\phi^n$ preserves the Levi-Civita connection of $g$, we have \[ \exp_x = \phi^{-n} \circ \exp_x \circ D_x \phi^n. \] Hence, $\exp_x$ is an isometry from $B(0, r_\phi^{-n} \epsilon)$ to $B_g(x, r_\phi^{-n} \epsilon)$ for all $n \geq 0$. Since the balls $B_g(x, r_\phi^{-n} \epsilon)$ cover $M$, $\exp_x$ is an isometry from $\mathbb R^q$ to $M$.
\end{proof}

\begin{prop} \label{propProperAction}
Consider a complete connected Riemannian manifold $(M,g)$. If $\mathrm{Sim}(M)$ does not act properly on $M$, then $M$ is (globally) isometric to $\mathbb R^q$ for some $q \geq 0$.
\end{prop}
\begin{proof}
Since $M$ is complete and connected, the isometry group $\mathrm{Isom}(M)$ acts properly on $M$. In the same way, if $\mathrm{Sim}(M)$ does not act properly on $M$, there exist a compact set $K \subseteq M$ and a sequence $(S_n)$ of similarities such that $K \cap S_n(K) \neq \emptyset$ and the ratio of $S_n$ (written by $r_n$) tends to $+\infty$ or $0$ when $n \to +\infty$. Considering $S_n^{-1}$ instead of $S_n$ if necessary, we may assume that $r_n \to 0$.

Let $K' = \setof{x \in M}{d(x, K) \leq \epsilon}$ for some small $\epsilon > 0$, where $d$ is the distance induced by $g$ in $M$. Then $S_n(K') = \setof{x \in M}{d(x, S_n(K)) \leq r_n \epsilon}$: in particular, for some large enough $n_0 > 0$, $S_{n_0}(K') \subseteq K'$. Thus, $S_{n_0}$ has a fixed point and $M$ is isometric to $\mathbb R^q$ by Proposition~\ref{propHomothetyFlat}.
\end{proof}

\begin{prop} \label{propProductProperActions}
Consider the product of two connected Riemannian manifolds, denoted by $(M, h) = (M_1, h_1) \times (M_2, h_2)$, and a subgroup $G$ of $\mathrm{Sim}(M)$ which preserves the product structure (\emph{i.e.} which is a subgroup of $\mathrm{Sim}(M_1) \times \mathrm{Sim}(M_2)$), and acts on $M$ in a cocompact way. Also assume that $\mathrm{Sim}(M)$ contains elements which are not isometries. Then, either $M_1 = \mathbb R^q$ or $M_2 = \mathbb R^q$, for some $q \geq 0$.
\end{prop}
\begin{proof}
Assume that the conclusion is false. In view of Proposition~\ref{propProperAction}, $\mathrm{Sim}(M_1)$ and $\mathrm{Sim}(M_2)$ act properly on $M_1$ and $M_2$ respectively.

Since $G$ acts cocompactly on $M$, there is a compact set $K \subseteq M$ such that $\mathrm{Sim}(M) \cdot K = M$. We may assume that $K = K_1 \times K_2$, where $\mathrm{Sim}(M_1) \cdot K_1 = M_1$ and $\mathrm{Sim}(M_2) \cdot K_2 = M_2$.

Choose $x_1 \in K_1$. Since $\mathrm{Sim}(M_1)$ acts properly on $M_1$, there is a constant $R > 1$ such that for all $\gamma \in \mathrm{Sim}(M_1)$ satisfying $\gamma(x_1) \in K_1$, the ratio of $\gamma$ is between $R$ and $1/R$. Likewise, choose $x_2 \in K_2$. There is a constant, still called $R$, such that the ratio of any $\gamma \in \mathrm{Sim}(M_2)$ satisfying $\gamma(x_2) \in K_2$ is between $R$ and $1/R$.

We assumed that $\mathrm{Sim}(M)$ contains elements which are not isometries, so there exists $\gamma_0 \in \mathrm{Sim}(M_1)$ whose ratio is greater than $R^3$. And since $G \cdot K = M$, there exists $\gamma = (\gamma_1, \gamma_2) \in G$ such that $\gamma(\gamma_0(x_1), x_2) \in K$. Then, $\gamma_1 \circ \gamma_0(x_1) \in K_1$, so the ratio of $\gamma_1 \circ \gamma_0$ is smaller than $R$: hence, the ratio of $\gamma_1$ is smaller than $1/R^2$. Meanwhile, $\gamma_2 (x_2) \in K_2$, so the ratio of $\gamma_2$ is greater than $1/R$. But since $(\gamma_1, \gamma_2) \in \mathrm{Sim}(M)$, $\gamma_1$ and $\gamma_2$ should have the same ratio, which is impossible.
\end{proof}

In the setting of Theorem~\ref{thmConformalStructure}, since the similarity structure induces a locally metric connection on $M$, Theorem~\ref{thmLocallyMetric} implies that $\tilde M$ admits a de Rham decomposition. Assuming that $\tilde M$ is the product of two manifolds $M_1$ and $M_2$, there is a finite index subgroup of $\pi_1(M)$ which preserves the product structure of $M$: it acts cocompactly on $M$ and contains elements which are not isometries. Thus, we may apply Proposition~\ref{propProductProperActions}, which completes the proof of Theorem~\ref{thmConformalStructure}.

\section{Closures of the leaves} \label{sectTori}
This section is devoted to the proof of Theorem~\ref{thmClosures}.

Consider $\pi_1(M) \subseteq \mathrm{Sim}(\mathbb R^q) \times \mathrm{Sim}(N)$ and define $P$ as the image of $\pi_1(M)$ by the projection onto the second factor, \emph{i.e.}
\[ P = \setof{ b \in \mathrm{Sim}(N) }{ \exists \ a \in \mathrm{Sim}(\mathbb R^q), \quad (a, b) \in \pi_1(M)}. \]
Denote by $\overline P$ the closure of $P$ in $\mathrm{Sim}(N)$, and by $\overline P^0$ the identity component of $\overline P$.

In Example~\ref{exThird}, $\overline P^0$ is the group $\mathbb R$ acting by translation on the first factor of $N = E^u \times (0, +\infty)$. In general, we will prove the following lemma:
\begin{lemme} \label{lemmaP0Abelian}
The group $\overline P^0$ is abelian.
\end{lemme}

\subsection{Generalities on lattices in Lie groups}

Here, we state classical general facts about lattices in Lie groups: in this paper, a lattice is a \emph{discrete, cocompact subgroup} of a Lie group.

\begin{prop} \label{propEquivLattice}
Let $G$ be a Lie group, $\Gamma$ a lattice in $G$, and $N$ a normal Lie subgroup of positive dimension in $G$. Then the following two properties are equivalent:
\begin{enumerate}
\item The group $\Gamma \cap N$ is a lattice in $N$.
\item The image of $\Gamma$ by the projection $G \to G / N$ is a lattice in $G / N$.
\end{enumerate}
\end{prop}
\begin{proof}
We have the following chain of equivalences: $\Gamma \cap N$ is a lattice in $N$
$\iff$ $N / \Gamma \cap N$ is compact
$\iff$ the image of $N$ in $G / \Gamma$ is compact
$\iff$ the image of $N$ in $G / \Gamma$ is closed in $G / \Gamma$
$\iff$ the subgroup of $G$ generated by $N \cup \Gamma$ in $G$ is closed
$\iff$ the image of $\Gamma$ by the projection $G \to G/N$ is closed in $G / N$
$\iff$ the image of $\Gamma$ by the projection $G \to G/N$ is a Lie subgroup of $G / N$
$\iff$ the image of $\Gamma$ by the projection $G \to G/N$ is discrete
$\iff$ the image of $\Gamma$ by the projection $G \to G / N$ is a lattice in $G / N$.
\end{proof}

Proposition~\ref{propEquivLattice} appears, for example, in~\cite{MR1016093}.

Propositions~\ref{propNeutral}, \ref{propNilpotentCenter} and~\ref{propRaghu} give sufficient conditions for the two equivalent assertions of Proposition~\ref{propEquivLattice} to hold true.

The following proposition is classical: since the proof is short, we recall it here.

\begin{prop} \label{propNeutral}
Let $G$ be a Lie group and $\Gamma$ a lattice in $G$. Then $\Gamma \cap G^0$ is a lattice in $G^0$ (where $G^0$ is the identity component of $G$).
\end{prop}
\begin{proof}
The group $\Gamma \cap G^0$ is discrete since it is a subgroup of the discrete group $\Gamma$. There remains to show that $\Gamma \cap G^0$ is cocompact in $G^0$.

The group $G^0$ acts on $G$ by left translation. The orbits are the connected components of $G$, so they are open. Thus, the orbits of the action of $G^0$ on $G/\Gamma$ by left translation are also open. Therefore, each orbit is closed (since its complement is a union of open orbits). In particular, the orbit $G^0 / (\Gamma \cap G^0)$ is closed in $G / \Gamma$ which is compact, so $G^0 / (\Gamma \cap G^0)$ is compact.
\end{proof}

The following proposition is also classical, see for example~\cite{MR0507234} page 40:

\begin{prop} \label{propNilpotentCenter} If $G$ is nilpotent and $\Gamma \subseteq G$ is a lattice, $\Gamma \cap Z(G)$ is a lattice in $Z(G)$, where $Z(G)$ is the center of $G$.
\end{prop}

We now recall Bieberbach's theorem (see~\cite{MR1511623}):
\begin{thm}[Bieberbach, 1911] \label{thmBieberbach}
Consider a Euclidean space $\mathbb R^q$, and a lattice $G$ in $\mathrm{Isom}(\mathbb R^q) = O(\mathbb R^q) \ltimes \mathbb R^q$. Then $G \cap \mathbb R^q$ is a lattice in $\mathbb R^q$.
\end{thm}

In order to state a more general version of Theorem~\ref{thmBieberbach}, we define the \emph{radical} and the \emph{nilradical} of a Lie group.

\begin{definition}
The \emph{radical} of a Lie group $G$, written by $\mathrm{Rad}(G)$, is the unique maximal normal connected closed solvable subgroup of $G$.
The \emph{nilradical} of a Lie group $G$, written by $\mathrm{Nil}(G)$, is the unique maximal normal connected closed nilpotent subgroup of $G$.
\end{definition}

\begin{prop} \label{propRaghu}
Let $G$ be a connected Lie group of the form $G = S \ltimes R$, where $R = \mathrm{Rad}(G)$ and $S$ is a semisimple Levi subgroup (such a decomposition always exists if $G$ is simply connected). Assume that the kernel of the action of $S$ on $R$ does not contain compact factors. If $\Gamma$ is a lattice in $G$, then $\Gamma \cap \mathrm{Nil}(G)$ is a lattice in $\mathrm{Nil}(G)$.
\end{prop}

Proposition~\ref{propRaghu} generalizes Theorem~\ref{thmBieberbach}. It was initially stated in~\cite{MR0507234} as a corollary of a theorem by Auslander~\cite{MR0123637}, but Raghunathan's proof turned out to be flawed. A correct proof is available in the appendix of~\cite{MR1038220}.

In order to apply Proposition~\ref{propRaghu}, the following will be useful:
\begin{prop} \label{propFactor}
Consider a connected, simply connected Lie group $G$ and its maximal normal compact subgroup $K$. Then there is a subgroup $L$ of $G$ such that $G = K \times L$.
\end{prop}
\begin{proof}
Since the group $G$ is simply connected, it admits a decomposition $G = S \ltimes R$, where $S$ is semisimple and $R = \mathrm{Rad}(G)$ is solvable. First, notice that $K \cap R = \{0\}$, because a solvable, simply connected Lie group cannot have any proper compact subgroup. Thus, the restriction to $K$ of the projection $G \to G / R$ is injective and hence, since it is normal in $G$,  $K$ projects isomorphically to a semi-simple compact factor of $G/R$. This implies that there exists a section $\sigma: G/R \to G$ whose image $S'$ contains $K$. Since $K$ is a normal subgroup of the semisimple group $S'$, there exists a subgroup $L'$ of $S'$ such that $S' = K \times L'$, and therefore $G = S' \ltimes R = (K \times L') \ltimes R$. Finally, since $K$ is normal in $G$, it cannot act on $R$, so $G = K \times (L' \ltimes R)$ and the proposition is proved.
\end{proof}
Notice that there always exists a maximal normal compact subgroup: if $K_1$ and $K_2$ are two normal compact subgroups, then $K_1 K_2$ is also a normal compact subgroup.

\subsection{Proof of Lemma~\ref{lemmaP0Abelian}}

\begin{lemme} \label{lemmaLieGroup}
The group $\overline P$ is a Lie group which acts properly on $N$.
\end{lemme}
\begin{proof}
With Theorem~\ref{thmFoliations}, there exists a transverse Riemannian structure on $\mathcal F$: it induces naturally a structure of complete Riemannian manifold on $N$, such that $P$ acts by isometries. Thus, $\overline P$ is a closed subgroup of the group of isometries of $N$ for this metric, so it is a Lie group which acts properly on $N$.
\end{proof}

\begin{lemme} \label{lemmaP0}
Every element of $\overline P$ is the product of an element of $P$ by an element of $\overline P^0$, \emph{i.e.} $\overline P = P \cdot \overline P^0$.
\end{lemme}
\begin{proof}
Since $P \subseteq \overline P$ and $\overline P^0 \subseteq \overline P$, we have $P \cdot \overline P^0 \subseteq \overline P$.

Conversely, consider some $p \in \overline P$, and $(p_n)$ a sequence of elements of $P$ converging to $p$ in $\overline P$. Then $(p_n)^{-1} p$ converges to the identity in $\overline P$. Since $\overline P$ is a Lie group, it is locally connected and therefore $(p_n)^{-1} p \in \overline P^0$ for a large enough $n$. Hence for this $n$, $p = p_n (p_n)^{-1} p \in P \cdot \overline P^0$.
\end{proof}

\begin{lemme} \label{lemmaClosure}
Denote by $f$ the covering $\mathbb R^q \times N \to M$, and let $(a, x) \in \mathbb R^q \times N$. Then the leaf of $\mathcal F$ containing $f(a, x)$ is $f(\mathbb R^q \times P x)$, and the closure of this leaf (in $M$) is $f(\mathbb R^q \times \overline P^0 x)$.
\end{lemme}
\begin{proof}
The leaf $\tilde{\mathcal F}$ (the foliation on the universal cover $\tilde M$) containing $(a,x)$ is $\mathbb R^q \times \{x\}$. Thus its projection by $f$ is $f(\mathbb R^q \times \{x\})$, which is equal to $f(\mathbb R^q \times P x)$ (because for any $\gamma \in \pi_1(M)$ and $y \in \tilde M$, $f(\gamma y) = f(y)$). Thus, the leaf of $\mathcal F$ containing $f(a, x)$ is $f(\mathbb R^q \times P x)$.

The closure of this leaf is $\overline{f(\mathbb R^q \times P x)}$. Let us show that $\overline{f(\mathbb R^q \times P x)} = f(\mathbb R^q \times \overline P x)$. If $f(a, p x) \in f(\mathbb R^q \times \overline P x)$, then there exists a sequence $(p_n)$ of elements of $P$ such that $p_n \to p$, and therefore $f(a, p_n x)$ is a sequence of elements of $f(\mathbb R^q \times P x)$ converging to $f(a, p x)$. Conversely, for $y \in \overline{f(\mathbb R^q \times P x)}$, there exists $y_n \in f(\mathbb R^q \times P x)$ such that $y_n \to y$. One may find $(a_n)$ a sequence in $\mathbb R^q$ and $p_n$ a sequence in $P$ such that $f(a_n, p_n x) = y_n$, and $(a_n, p_n x)$ converges in $\mathbb R^q \times N$ to some point $(a, b)$. Furthermore, since $\overline P$ acts properly on $N$ (Lemma~\ref{lemmaLieGroup}), one may assume (up to extraction) that $(p_n)$ converges to some $p \in \overline P$. Hence, $y = f(a, px)$.

Finally, $f(\mathbb R^q \times \overline P^0 x) = f(\mathbb R^q \times P \cdot \overline P^0 x) = f(\mathbb R^q \times \overline P x)$ by Lemma~\ref{lemmaP0}.
\end{proof}

Now let us notice the following:
\begin{lemme} \label{lemmaLattice}
The group $\pi_1(M)$ is a lattice in the Lie group $\mathrm{Sim}(\tilde M) \cap (\mathrm{Sim}(\mathbb R^q) \times \overline P)$.
\end{lemme}
\begin{proof}
We will write $S = \mathrm{Sim}(\tilde M) \cap (\mathrm{Sim}(\mathbb R^q) \times \overline P)$. Notice that $S$ is a closed subgroup of the Lie group $\mathrm{Sim}(\mathbb R^q) \times \overline P$, so it is a Lie group.

The group $\pi_1(M)$ is discrete, so there remains to show that it is cocompact.

Since $M$ is compact, there is a compact set $K_1 \subseteq \tilde M$ such that $\pi_1(M) \cdot K_1 = \tilde M$. Define \[ K_2 = \setof{\phi \in S}{\phi(K_1) \cap K_1 \neq \emptyset}. \] The set $K_2$ is compact because the action of $\mathrm{Sim}(\tilde M)$ is proper (by Proposition~\ref{propProperAction}). Then for all $\psi \in S$ there exists $\gamma \in \pi_1(M)$ such that $\gamma (\psi (K_1)) \cap K_1 \neq \emptyset$: hence $\gamma \circ \psi \in K_2$. This proves that $\pi_1(M) \cdot K_2 =  S$, and therefore $\pi_1(M)$ is cocompact in $S$.
\end{proof}

We denote by $\mathrm{Isom}^+(\mathbb R^q)$ (\emph{resp.} $\mathrm{Sim}^+(\mathbb R^q)$) the group of orientation-preserving isometries (\emph{resp.} similarities) of $\mathbb R^q$.

Writing $P_I = \overline P^0 \cap \mathrm{Isom}(N)$, we have an exact sequence:
\[ 0 \to P_I \to \overline P^0 \xrightarrow{r} \mathbb R_{>0} \]
where $r : \overline P^0 \to \mathbb R_{>0}$ gives the ratio of a similarity. Since $\overline P^0$ is connected, the mapping $r$ is either surjective or constantly equal to $1$. Furthermore, if $r$ is surjective, it has a section because $\overline P^0$ is a Lie group. Thus, writing $H \subseteq \mathbb R_{>0}$ the image of $r$, we may write, up to isomorphism, $\overline P^0 = H \ltimes P_I$ (in particular, $P_I$ is connected). In addition, $\mathrm{Sim}^+(\mathbb R^q) = \mathbb R_{>0} \ltimes \mathrm{Isom}^+(\mathbb R^q)$. Considering the group $T = \mathrm{Sim}(\tilde M) \cap (\mathrm{Sim}^+(\mathbb R^q) \times \overline P^0)$, we may write:
\[ T = H \ltimes ( \mathrm{Isom}^+(\mathbb R^q) \times P_I). \]

Denoting by $\tilde P_I$ the universal cover of $P_I$, and by $\tilde T$ the universal cover of $T$, we obtain:
\[ \tilde T = H \ltimes ((\tilde{SO}(q) \ltimes \mathbb R^q) \times \tilde P_I), \]
where $\tilde{SO}(q)$ is the universal cover of $SO(q)$.

The group $\pi_1(M)$ acts as a subgroup of $\mathrm{Sim}(\tilde M) \cap (\mathrm{Sim}(\mathbb R^q) \times \overline P)$. This subgroup is a lattice (see Lemma~\ref{lemmaLattice}), so by Proposition~\ref{propNeutral}, $\pi_1(M) \cap T$ is also a lattice in $T$.

Consider $\Gamma$ the subgroup of $\tilde T$ defined as the pullback of $\pi_1(M) \cap T$ by the covering $\tilde T \to T$: it is a lattice in $\tilde T$. The image of $\pi_1(M) \cap T$ by the projection onto the second factor $\mathrm{Sim}^+(\mathbb R^q) \times \overline P^0 \to \overline P^0$ is $P \cap \overline P^0$ (by definition of $P$). Since $P \cap \overline P^0$ is dense in $\overline P^0$, this implies that the image of $\Gamma$ by the projection onto $H \ltimes \tilde P_I$ is dense in $H \ltimes \tilde P_I$.

Let $K$ be the maximal normal compact connected subgroup of $\tilde P_I$ and write $\tilde P_I = K \times L$ (using Proposition~\ref{propFactor}). Then $H \ltimes K$ is isomorphic to the direct product $H \times K$, so \[ \tilde T = K \times \left(H \ltimes ((\tilde{SO}(q) \ltimes \mathbb R^q) \times L)\right). \]

\begin{lemme} \label{lemmaNilpotent}
The group $L$ is nilpotent and the group $H$ is trivial. In particular, $\overline P^0$ acts on $N$ by isometries.
\end{lemme}
\begin{proof}
Since $K$ is compact, the image $\Gamma_1$ of $\Gamma$ by the projection onto $H \ltimes ((\tilde{SO}(q) \ltimes \mathbb R^q) \times L)$ is a lattice.

The nilradical of $\mathrm{Isom}^+(\mathbb R^q) = \tilde{SO}(q) \ltimes \mathbb R^q$ is $\mathbb R^q$.

Thus, the nilradical of the Lie group $H \ltimes ((\tilde{SO}(q) \ltimes \mathbb R^q) \times L)$ contains $\mathbb R^q$. Since $H$ acts by homotheties, the action of a nontrivial element of $H$ on $\mathbb R^q$ is not unipotent, and therefore the image of the nilradical by the projection onto $H$ is trivial. Hence, the nilradical of the Lie group $H \ltimes ((\tilde{SO}(q) \ltimes \mathbb R^q) \times L)$ is $\mathbb R^q \times \mathrm{Nil}(L)$, where $\mathrm{Nil}(L)$ is the nilradical of $L$.

Moreover, $H \ltimes ((\tilde{SO}(q) \ltimes \mathbb R^q) \times L)$ does not contain any nontrivial normal compact subgroup, because $K$ is defined as the maximal normal compact connected subgroup of $\tilde P_I$. Hence one may apply Proposition~\ref{propRaghu} and Proposition~\ref{propEquivLattice}, which show that the image of $\Gamma_1$ by the projection onto $H \ltimes (\tilde{SO}(q) \times L / \mathrm{Nil}(L))$ is a lattice. The assumptions of Proposition~\ref{propRaghu} are satisfied because $K$ is the maximal normal compact connected subgroup of $\tilde P_I$. Furthermore, $\Gamma_1$ contains the fundamental group of $SO(q)$, so the image of $\Gamma_1$ by the projection onto $H \ltimes (SO(q) \times L/\mathrm{Nil}(L))$ is a lattice.

Since $SO(q)$ is compact, the image of $\Gamma_1$ by the projection onto $H \ltimes (L / \mathrm{Nil}(L))$ is a lattice. But this image is dense. Therefore, $H$ and $L/\mathrm{Nil}(L)$ are discrete. Since they are connected, $H = \{1\}$ and $L = \mathrm{Nil}(L)$.
\end{proof}

\begin{lemme} \label{lemmaLAbelian}
The group $L$ is abelian.
\end{lemme}
\begin{proof}
Consider $\Gamma_2$ the intersection of $\Gamma_1$ with $\mathbb R^q \times L$. We use again the fact that the image of $\Gamma_1$ by the projection onto $H \ltimes (SO(q) \times L / \mathrm{Nil}(L))$ is a lattice: it now means that the image of $\Gamma_1$ by the projection onto $SO(q)$ is a lattice, and therefore finite. Since the image of $\Gamma_1$ by the projection onto the second factor $L$ is dense, and $\Gamma_2$ has finite index in $\Gamma_1$, this implies that the image of $\Gamma_2$ by the projection onto $L$ is dense. In addition, still with Proposition~\ref{propRaghu}, $\Gamma_2$ is a lattice in $\mathbb R^q \times L$.

By Proposition~\ref{propNilpotentCenter}, the image of $\Gamma_2$ by the projection onto $L / Z(L)$ is a lattice in $L / Z(L)$, where $Z(L)$ is the center of $L$. This image is also dense, so $L/Z(L)$ is discrete. Since $L$ is connected, $L = Z(L)$ and $L$ is abelian.
\end{proof}

\begin{lemme} \label{lemmaKAbelian}
The group $K$ is abelian.
\end{lemme}
\begin{proof}
Consider the group $\Gamma_3 = \Gamma \cap (K \times \mathbb R^q \times L)$, and $\Gamma_4$ the image of $\Gamma_3$ by the projection onto $K$. Then for all $k, k' \in K$ and $\gamma, \gamma' \in \Gamma_2$, we may write \[ [(k, \gamma), (k', \gamma')] = ([k, k'], [\gamma, \gamma']) = ([k, k'], 1). \] Thus $[\Gamma_4, \Gamma_4] \subseteq \Gamma_3 \cap K = \Gamma \cap K$.

Since we already know that the image of $\Gamma$ by the projection onto $\tilde{SO}(q)$ is finite, we deduce that $\Gamma_3$ is a subgroup of $\Gamma$ of finite index, and therefore $\Gamma_4$ is dense in $K$, which implies that $[K, K] \subseteq \Gamma \cap K$. But $K$ is connected so $[K, K]$ is trivial: $K$ is abelian (in fact, $K$ is a compact abelian simply connected Lie group, so it is trivial).
\end{proof}

Lemmas~\ref{lemmaNilpotent}, \ref{lemmaLAbelian} and~\ref{lemmaKAbelian} together imply that $\tilde P_I$ is abelian, and therefore Lemma~\ref{lemmaP0Abelian} is proved.

\subsection{End of the proof of Theorem~\ref{thmClosures}} \label{endProofTori}

Denote by $\overline F$ the closure of a leaf in $M$. By Lemma~\ref{lemmaClosure}, $\overline F = f(\mathbb R^q \times \overline P^0 x)$. We already know that $\overline F$ is a submanifold of $M$ (see~\cite{MR932463}). Lemma~\ref{lemmaP0Abelian} implies that $\mathbb R^q \times \overline P^0 x$ is isometric to the product of a Euclidean space by a flat torus, so $\overline F$ is flat. Moreover,

\begin{prop} \label{propRiemannianManifolds}
The closures of the leaves are Riemannian manifolds, \emph{i.e.} the Riemannian metric on the universal cover $\tilde M$ induces a Riemannian metric on $\overline F$.
\end{prop}
\begin{proof}
$\mathbb R^q \times \overline P^0 x$ is complete, so every similarity of $\mathbb R^q \times \overline P^0 x$ of ratio $\neq 1$ has a fixed point (by the Banach fixed point theorem). Thus, the elements of $\pi_1(M)$ with ratio $\neq 1$ act freely on $N / \overline P^0$, which proves the proposition.
\end{proof}

To study the dimension of $\overline F$, we will need the following lemma:

\begin{lemme} \label{lemmaFreeAction}
The group $\pi_1(M)$ acts freely on $N$.
\end{lemme}
\begin{proof}
Consider some $u \in \pi_1(M)$ with $u \neq \mathrm{Id}$, and write $u = (u', u'')$, where $u' \in \mathrm{Sim}(\mathbb R^q)$ and $u'' \in \mathrm{Sim}(N)$. Assume that $u''$ has a fixed point $a \in N$. Then $u'$ has no fixed point (because $\pi_1(M)$ acts freely on $\tilde M$). Therefore, $u'$ is an isometry of $\mathbb R^q$, so one may write $u'(x) = R_u x + t_u$ for $x \in \mathbb R^q$, where $R_u \in O(\mathbb R^q)$ and $t_u \in \mathbb R^q$, with $t_u \neq 0$.

Now, consider $v \in \pi_1(M)$ with ratio $\lambda \in (0,1)$, and write $v = (v', v'')$, where $v' \in \mathrm{Sim}(\mathbb R^q)$ and $v'' \in \mathrm{Sim}(N)$. We have $v'(x) = \lambda R_v x + t_v$ for $x \in \mathbb R^q$, where $R_v \in O(\mathbb R^q)$ and $t_v \in \mathbb R^q$. Since $v'$ has a fixed point in $\mathbb R^q$, we may apply a translation in $\mathbb R^q$ in order to assume that $t_v = 0$.

Now for all $k \in \mathbb N$ and $x \in \mathbb R^q$, we have \[ (v')^k (u') (v')^{-k} (x) = R_v^k R_u R_v^{-k} x + \lambda^k R_v^k t_u. \]
Furthermore, $(v'')^k u'' (v'')^{-k}$ has a fixed point because $u''$ has a fixed point. Thus, $\setof{v^k u v^{-k}}{k \in \mathbb N}$ is an infinite, relatively compact subset of $\pi_1(M)$, which contradicts the fact that $\pi_1(M)$ is discrete.
\end{proof}

We write $d = \mathrm{dim}(\overline F)$. Then $\mathcal F$ induces a foliation of dimension $q$ on $\overline F$, so $q \leq d \leq q + n$.

If $\overline F$ has dimension $q$, then $\mathcal F$ has codimension $0$ in $F$, so $\overline F = F$, and therefore $F$ is compact. But by Lemma~\ref{lemmaFreeAction}, $F$ is homeomorphic to $\mathbb R^q$, so this is impossible.

If $\overline F$ has dimension $q + n$, then $\overline F$ is open and closed in $M$. Since $M$ is connected, we have $\overline F = M$, which contradicts the fact that $M$ is not flat.

Therefore, $q < d < q + n$ and Theorem~\ref{thmClosures} is proved.

\subsection{End of the proof of Theorem~\ref{thmTori}} Define $\Gamma_0 = \pi_1(M) \cap \left(\mathrm{Sim}(\mathbb R^q) \times \overline P^0\right)$. For all $x \in N$, consider the following subgroup of $\pi_1(M)$: \[ S_x = \setof{p \in \pi_1(M)}{p \cdot x \in \overline P^0 x}. \]

\begin{lemme} \label{lemmaGamma0} The group $\Gamma_0$ is contained in  $\mathbb R^q \times \overline P^0$. Moreover, it is a lattice in $\mathbb R^q \times \overline P^0$.
\end{lemme}
\begin{proof}
Since $\overline P^0$ is abelian, for all $u, v \in \Gamma_0$, the commutator $uvu^{-1}v^{-1}$ acts trivially on $N$. By Lemma~\ref{lemmaFreeAction}, $uvu^{-1}v^{-1} = \mathrm{Id}$. Therefore, $\Gamma_0$ is abelian.

The closure of a leaf $f(\mathbb R^q \times \overline P^0 x)$ (where $x \in N$) is isomorphic to $\left(\mathbb R^q \times \overline P^0 x\right) / S_x$. In particular, since the closures of the leaves are closed Riemannian manifolds (by Proposition~\ref{propRiemannianManifolds}), $S_x$ acts cocompactly by isometries on $\mathbb R^q \times \overline P^0 x$. Also recall that $\mathbb R^q \times \overline P^0 x$ is the product of a Euclidean space by a flat torus (Lemma~\ref{lemmaP0Abelian}). By applying Theorem~\ref{thmBieberbach} to the universal cover of $\mathbb R^q \times \overline P^0 x$, we deduce that $S_x \cap (\mathbb R^q \times \overline P^0)$ (which is a subgroup of $\Gamma_0$) also acts cocompactly. In particular, the image of $\Gamma_0$ by the projection onto $\mathrm{Sim}(\mathbb R^q)$ contains translations: since it is abelian, it contains only translations. Thus, $\Gamma_0$ is a lattice in $\mathbb R^q \times \overline P^0$.
\end{proof}


Now, consider the representation $\rho: \pi_1(M) \to \mathrm{Aut}(\Gamma_0)$ given by the action of $\pi_1(M)$ onto $\Gamma_0$ by conjugation.

\begin{lemme} \label{lemmaFiniteIndex}  There is a subgroup $J \subseteq \pi_1(M)$ of finite index such that $\rho(J)$ has no torsion.
\end{lemme}
\begin{proof}
The group $\Gamma_0$ is a lattice in $\mathbb R^q \times \overline P^0$, which implies that $\Gamma_0$ is a finitely generated abelian group. Therefore, $\mathrm{Aut}(\Gamma_0)$ is a subgroup of $GL_m(\mathbb Z)$ for some $m \in \mathbb Z$. By Selberg's lemma (see for example~\cite{MR925989}), there is a subgroup of $\rho(\pi_1(M))$ of finite index which is torsion-free: the preimage $J$ of this subgroup by $\rho$ is the desired group.
\end{proof}

In the following, we fix such a subgroup $J \subseteq \pi_1(M)$.

\begin{lemme} \label{lemmaSubgroupGamma0}
For all $x \in N$, the group $S_x \cap J$ is a subgroup of $\Gamma_0$.
\end{lemme}
\begin{proof}
Choose $a \in S_x \cap J$. There exists an element $t \in \mathbb R^q \times \overline P^0$ such that the action of $ta$ on $\mathbb R^q \times N$ has a fixed point. Consider the subgroup $H$ of $\mathbb R^q \times \overline P$ generated by $ta$: by Proposition~\ref{propProperAction}, $H$ is relatively compact in $\mathbb R^q \times \overline P$. Then the image of $H$ by the projection $\mathbb R^q \times \overline P \to \overline P / \overline P^0$ is discrete and relatively compact, so it is finite. Thus, there exists $n \geq 1$ such that $(ta)^n \in \mathbb R^q \times \overline P^0$, which implies that $\rho(a^n)$ is trivial. Since $\rho(J)$ has no torsion, $\rho(a)$ is trivial.

We have shown that all the elements of $S_x \cap J$ act trivially on $\overline P^0$ by conjugation, which implies that $S_x \cap J \subseteq \mathrm{Sim}(\mathbb R^q) \times \overline P^0$.
\end{proof}

Since $J$ has finite index in $\pi_1(M)$, there is a finite covering $\tilde M / J \to M$: we will write $M' = \tilde M / J$ and show that the closures of the leaves of the foliation $\mathcal F'$ (induced by $\mathcal F$ on $M'$) are tori. Denote by $\overline {F'}$ the closure of a leaf of $\mathcal F'$. By Lemma~\ref{lemmaClosure}, $\overline {F'} = f'(\mathbb R^q \times \overline {P'}^0 x)$ for some $x \in N$, where $f'$ is the projection $\tilde M \to M'$ and $P'$ is the image of $J$ by the projection onto $\mathrm{Sim}(N)$. Since $P'$ has finite index in $P$, $\overline{P'}$ has finite index in $\overline P$ and therefore $\overline {P'}^0 = \overline P^0$, so $\overline {F'} = f'(\mathbb R^q \times \overline P^0 x)$: in other words, $\overline {F'} = (\mathbb R^q \times \overline P^0 x) / (S_x \cap J)$. Since $S_x \cap J$ is a subgroup of $\mathbb R^q \times \overline P^0$ (by Lemma~\ref{lemmaSubgroupGamma0} and Lemma~\ref{lemmaGamma0}), the group $(\mathbb R^q \times \overline P^0) / (S_x \cap J)$, which is the product of a linear space by a torus, acts transitively on $\overline {F'}$ by isometries. Since $\overline {F'}$ is compact, it is isometric to a flat torus, which ends the proof of Theorem~\ref{thmTori}.

\section{Classification in dimension 2} \label{sectDim2}
In this section, we assume that $\dim N = 2$ and prove Theorem~\ref{thmDim2}. Since $N$ is non-complete and simply connected, it is diffeomorphic to $\mathbb R^2$.

\begin{lemme}
The group $\overline P^0$ acts freely on $N$.
\end{lemme}
\begin{proof}
We assume that there exists $p \in \overline P^0$ which has a fixed point $a \in N$ and look for a contradiction. Consider a one-parameter subgroup $G \subseteq \overline P^0$ which contains $a$. Then the flow induced by $G$ on $N$ has a closed orbit in $\mathbb R^2$, so it has a fixed point $x_0 \in N$ (here we use the fact that $\mathrm{dim}(N) = 2$). Thus, the closure $\overline G$ of $G$ in $\overline P^0$ fixes $x_0$. Since $\overline P^0$ acts properly, $\overline G$ is compact. Since $\overline G$ is abelian (by Lemma~\ref{lemmaP0Abelian}), it is a torus: it contains a closed Lie subgroup $H$ which is isomorphic to $\mathbb R / \mathbb Z$, and which fixes $x_0$.

Choose a point $x_1$ which is not fixed by $H$. Then $Hx_1$ defines a closed curve in $N$. Thus, $H$ has a fixed point $x_0' \in N$ such that the curve $Hx_1$ in $N \setminus \{x_0'\}$ is not homotopic to a constant. Since $N$ is connected, the homotopy class of the curve $Hx$ in $N \setminus \{x_0'\}$ does not depend on the choice of $x \in N \setminus \{x_0'\}$. Therefore, $H$ has only one fixed point $x_0' = x_0$.

From now on, denote by $K(x)$ the curvature of $N$ at $x \in N$. The curvature $K$ is not constant on $N \setminus \{x_0\}$ because $N$ is a non-flat manifold with similarities: thus there is some $x_2 \in N \setminus \{x_0\}$ such that $K(x_1) \neq K(x_2)$. Assume that the curve $Hx_2$ is in the unbounded component of the complement of $Hx_1$ (this is always possible up to a permutation of $x_1$ and $x_2$). Then the connected component $C$ of $N \setminus Hx_2$ containing $x_1$ is bounded.

Choose a similarity $h \in P$. Since $P$ acts properly, $h^n C \cap C = \emptyset$ for some $n \in \mathbb N$. But $H(h^n x_1)$ (the orbit of $h^n x_1$ under the flow $H$) does not intersect $H(h^n x_2)$ (because $K(x_1) \neq K(x_2)$), so it is contained in $h^n(C)$. Then $H(h^n x_1)$ is homotopic to a constant in $N \setminus \{x_0\}$: this contradiction ends the proof of the lemma.
\end{proof}

\begin{lemme}
The Lie group $\overline P^0$ is isomorphic to $\mathbb R$.
\end{lemme}
\begin{proof}
For all $x \in N$, Theorem~\ref{thmClosures} implies that $\mathbb R^q \times \overline P^0 x$ has dimension $q + 1$, so $\overline P^0 x$ has dimension $1$. Since $\overline P^0$ acts freely and transitively on $\overline P^0 x$, we deduce that $\overline P^0$ is diffeomorphic to $\overline P^0 x$, so it has dimension $1$. Furthermore, $\mathbb R / \mathbb Z$ cannot act freely on the plane (if a flow has a closed orbit, then it has a fixed point), so $\overline P^0$ is isomorphic to $\mathbb R$.
\end{proof}

In the following, we fix an identification of $\overline P^0$ with $\mathbb R$.

\begin{lemme} \label{lemmeYZ}
There exists an open interval $(a, b)$ of $\mathbb R$ and a diffeomorphism $\Phi : \mathbb R \times (a, b) \to N$ which provides a system of coordinates $(y, z)$ on $N$ (with $y \in \mathbb R$, $z \in (a, b)$) such that:

\begin{itemize}
\item Writing $y \in \mathbb R$ and $z \in (a, b)$ the coordinates in $N$ given by $\Phi$, the action of $p \in \overline P^0$ on $N$ is $p(y, z) = (y + p, z)$.
\item The metric on $N$ is given by $\varphi(z) dy^2 + dz^2$, where $\varphi : (a, b) \to \mathbb R_{> 0}$ is a smooth function.
\end{itemize}
\end{lemme}
\begin{proof}
The group $\overline P^0$ acts freely and properly on $N$. Thus, $N / \overline P^0$ has a natural structure of smooth manifold of dimension $1$. Furthermore, $N / \overline P^0$ is simply connected because $N$ is, so $N / \overline P^0$ is diffeomorphic to $\mathbb R$. Thus, $N \to N / \overline P^0$ is a fiber bundle over a contractible space, so it is trivial. Up to isometry, we may write $N = (\mathbb R^2, g_N)$, where $g_N$ is a Riemannian metric on $\mathbb R^2$, and $\overline P^0$ acts by translation on the first coordinate.

Denote by $\mathcal F_N$ the foliation induced by the submersion $s : N \to N / \overline P^0$. Then $\mathcal F_N$ is the foliation induced on $N$ by the closures of the leaves of $\mathcal F$: it is the standard foliation of $\mathbb R^2$ by horizontal lines. Consider a vector field $X$ on $N$ (with the above identification, $X : \mathbb R^2 \to \mathbb R^2$) orthogonal to foliation $\mathcal F_N$ for the metric $g_N$, such that all vectors have length $1$ for the metric $g_N$. Denoting by $X_1$ and $X_2$ the two coordinates of $X$ in $\mathbb R^2$, we may assume that $X_2$ is everywhere positive. Let $\gamma : (a, b) \to \mathbb R^2$ be a maximal integral curve of $X$, where $(a, b)$ is an open interval of $\mathbb R$, and denote by $\gamma_1$ and $\gamma_2$ the two coordinates of $\gamma$ in $\mathbb R^2$. Notice that $\gamma_2$ is increasing, so $\lim_{t \to b} \gamma_2(t)$ exists. If it is finite, $X(\gamma(t))$ (which depends only on $\gamma_2(t)$) has a limit when $t \to b$, so $b = +\infty$, but the limit of $X_2(\gamma(t))$ is positive, which contradicts the fact that $\lim_{t \to b} \gamma_2(t)$ is finite. Thus, $\lim_{t \to b} \gamma_2(t) = +\infty$ and for the same reason, $\lim_{t \to a} \gamma_2(t) = - \infty$.

The mapping
\[ \begin{aligned} \Phi: \overline P^0 \times (a, b) & \to N
\\ (p, t) & \mapsto p \cdot \gamma(t) \end{aligned} \]
is bijective, and by the inverse function theorem, it is a diffeomorphism. Since $\overline P^0$ is the Lie group $\mathbb R$, this is the diffeomorphism announced in the statement of Lemma~\ref{lemmeYZ}. With these coordinates, the metric on $N$ is
\[ g = \alpha_y(y, z) dy^2 + \alpha_z(y, z) dz^2 + \alpha_{yz} dy dz. \]
For all $p \in \overline P^0$, the curve $p \circ \gamma$ has unit speed and is orthogonal to $\mathcal F_N$, so $\alpha_z$ is everywhere $1$ and $\alpha_{yz}$ is everywhere $0$. Also, the action of $\overline P^0$ implies that $\alpha_y$ depends only on $z$. Thus, the metric has the desired form.
\end{proof}

\begin{lemme} \label{lemmaDiscreteImage}
Consider the mapping $r : \overline P \to \mathbb R_{>0}$, which gives the ratio of a similarity. Then the image of $\overline P$ by $r$ is discrete.
\end{lemme}
\begin{proof}
If the image of $\overline P$ is not discrete, then $r$ is surjective, and therefore $r$ has a section (because $\overline P$ is a Lie group): but then, the image of $\overline P^0$ by $r$ is $\mathbb R$, which contradicts the fact that $\overline P^0$ contains only isometries.
\end{proof}

Now, each point of $\tilde M = \mathbb R^q \times N$ has three coordinates $(x, y, z)$, where $x \in \mathbb R^q$. We will denote by $g_{\mathbb R^q}$ the standard Euclidean metric in $\mathbb R^q$.

\begin{lemme} \label{lemmaP}
Choose any $p \in \pi_1(M)$ with ratio $\lambda \neq 1$. Then:
\begin{itemize}
	\item In Lemma~\ref{lemmeYZ}, one may require that:
	\begin{itemize}
		\item $(a, b) = (0, +\infty)$.
		\item For all $z \in (0, +\infty)$, $\varphi(\lambda(z)) = \lambda^{2q+2} \varphi(z)$.
	\end{itemize}
	\item The mapping $p$ has the form $p(x, y, z) = (p_0(x), p_1(y), p_2(z))$, where $p_0$ is a similarity of $\mathbb R^q$ of ratio $\lambda$, $p_1(y) = \lambda^{-q} y$ or $p_1(y) = - \lambda^{-q} y$, and $p_2(z) = \lambda z$.
\end{itemize}
\end{lemme}
\begin{proof}
Since $\overline P^0$ is normal in $\overline P$, the group $\overline P$ preserves the foliation $\mathcal F_N$, so the action of $p$ on $N$ preserves the foliation $\mathcal F_N$. Thus, $p$ also preserves $(\mathcal F_N)^\bot$, and we may write $p(x, y, z) = (p_0(x), p_1(y), p_2(z))$. Since $p$ has ratio $\lambda$, we have \[ p^*(g_{\mathbb R^q} + \varphi(z) dy^2 + dz^2) = \lambda^2 (g_{\mathbb R^q} + \varphi (z) dy^2 + dz^2), \] but also \[ p^*(g_{\mathbb R^q} + \varphi(z) dy^2 + dz^2) = p_0^* g_{\mathbb R^q} + (p_1'(y))^2 (\varphi(p_2(z)) dy^2 + (p_2'(z))^2 dz^2. \] Thus, for all $(x, y, z) \in \mathbb R^q \times \overline P^0 \times \mathbb (a, b)$, we have 
$(p_1'(y))^2 = \lambda^2 \varphi(z) / \varphi(p_2(z))$.
In particular, $p_1'(y)$ does not depend on $y$: we will write $p_1'(y) = \mu$. Hence $p_1$ is a similarity of $\mathbb R$ of ratio $\mu$. Furthermore, $p_0$ is a similarity of $\mathbb R^q$ of ratio $\lambda$, and $p_2$ is a similarity of $(a, b)$ of ratio $\lambda$. Since $\Gamma_0$ is normal in $\pi_1(M)$, the mapping $p$ induces a diffeomorphism on $\mathbb R^q \times \overline P^0 / \Gamma_0$, which is compact by Lemma~\ref{lemmaGamma0}. Thus the mapping $(x, y) \mapsto (p_1(x), p_2(y))$ is volume-preserving, which implies that $\abs{\lambda^q \mu} = 1$.

The similarity $p_1$ has ratio $\mu \neq 1$, and therefore it has a fixed point in $\mathbb R$: up to a translation we may assume that this fixed point is $0$, thus $p_1(y) = \pm \lambda^{-q} y$. Since $P$ acts freely on $N$, this implies that $p_2$ has no fixed point, hence $(a, b) \neq \mathbb R$. Furthermore, $(a, b)$ cannot have finite length. Thus, $(a, b)$ is a half-line: it is isometric to $(0, +\infty)$. From now on, we will assume $(a, b) = (0, +\infty)$. Hence $p_2(z) = \lambda z$.

Considering again the equality $(p_1'(y))^2 = \lambda^2 \varphi(z) / \varphi(p_2(z))$, we deduce that for all $z \in (0, +\infty)$, $\varphi(\lambda z) = (\lambda / \mu)^2 \varphi(z)$.
\end{proof}

\begin{lemme} \label{lemmaPi1generators}
Choose $p \in \pi_1(M)$ which has ratio $\lambda < 1$, where $\lambda$ is maximal (this is made possible Lemma~\ref{lemmaDiscreteImage}). Then $\pi_1(M)$ is generated by $\Gamma_0$ and $p$.
\end{lemme}
\begin{proof}
Choose another element $\hat p \in \pi_1(M)$, $\hat p = (\hat p_0, \hat p_1, \hat p_2)$, with ratio $\hat \lambda$. Then there exists $k \in \mathbb Z$ such that $\hat \lambda = \lambda^k$, and therefore for all $z \in (0, +\infty)$, $p_2^{-k} \hat p_2 (z) = z$. Since $\pi_1(M)$ acts freely on $N$ (by Lemma~\ref{lemmaFreeAction}), $p_1^{-k} \hat p_1$ has no fixed point, so it is a translation, which means that $p^{-k} \hat p \in \Gamma_0$.
\end{proof}

Finally, apply a linear map in order to assume that $\Gamma_0$ is the lattice $\mathbb Z^{q+1}$ in $\mathbb R^{q+1}$. Then Lemmas~\ref{lemmaP} and~\ref{lemmaPi1generators} imply Theorem~\ref{thmDim2}.

\section*{Acknowledgements}
This paper contains results which were obtained during my PhD thesis: I would like to thank my advisor Abdelghani Zeghib for his help.

This work was also supported by the ERC Avanced Grant 320939, Geometry and Topology of Open Manifolds (GETOM).

I would like to thank the referees for helping me improve the paper.

\bibliographystyle{alpha}
\bibliography{ref}

\end{document}